\documentclass{amsart}
\usepackage[T1]{fontenc}
\usepackage[utf8]{inputenc}

\usepackage{algorithm}
\usepackage{algorithmic}

\usepackage[colorlinks=True, citecolor=blue, linkcolor=blue, urlcolor=blue]{hyperref}
\usepackage[nameinlink]{cleveref}
\usepackage{amsmath}
\usepackage{amssymb}
\usepackage{tikz}
\usepackage{tikz-cd}
\usepackage{booktabs}
\usepackage{comment}

\newenvironment{Cases}{\left.\begin{cases}}{\end{cases}\right\}}

\theoremstyle{definition}
\newtheorem{definition}{Definition}[section]
\newtheorem{example}[definition]{Example}
\newtheorem{remark}[definition]{Remark}

\theoremstyle{plain}
\newtheorem{theorem}[definition]{Theorem}
\newtheorem{lemma}[definition]{Lemma}
\newtheorem{proposition}[definition]{Proposition}
\newtheorem{corollary}[definition]{Corollary}

\DeclareMathOperator{\diag}{diag}

\DeclareMathOperator{\id}{id}

\DeclareMathOperator{\Hom}{Hom}

\DeclareMathOperator{\GL}{GL}

\newcommand{\QQ}{\mathbb{Q}}
\newcommand{\FF}{\mathbb{F}}

\newcommand{\ZZ}{\mathbb{Z}}
\newcommand{\NN}{\mathbb{N}}

\newcommand{\IQ}{\mathbb{Q}}

\newcommand{\IZ}{\mathbb{Z}}

\newcommand{\bh}{\mathbf{h}}
\newcommand{\bu}{\mathbf{u}}
\newcommand{\bv}{\mathbf{v}}
\newcommand{\bw}{\mathbf{w}}

\makeatletter
\newcommand*{\defeq}{\mathrel{\rlap{%
                     \raisebox{0.3ex}{\(\m@th\cdot\)}}%
                     \raisebox{-0.3ex}{\(\m@th\cdot\)}}%
                     =}
\makeatother

\title{Hensel lifting algorithms for quadratic forms}

\author{Simon Brandhorst}
\address{Universit\"at des Saarlandes, Campus E2.4, 66123 Saarbr\"ucken, Germany}
\email{brandhorst@math.uni-sb.de}

\author{Davide Cesare Veniani}
\address{Universität Stuttgart \\ Pfaffenwaldring 57 \\ 70569 Stuttgart \\ Germany}
\email{davide.veniani@mathematik.uni-stuttgart.de}

\date{\today}
\thanks{
S.B. is funded by the Deutsche Forschungsgemeinschaft (DFG, German Research Foundation) -- Project-ID 286237555 -- TRR 195.}

\newcommand\smvee{\raise0.9ex\hbox{\(\scriptscriptstyle\vee\)}}

\begin{document}

\begin{abstract}
We provide an algorithm to compute generators of the orthogonal group of the discriminant group associated to an integral quadratic lattice over the integers.
We give a closed formula for its order.
\end{abstract}

\maketitle

\section{Introduction}
Let \(L\) be an integral quadratic lattice defined over the ring of integers \(\IZ\). 
The discriminant group of~\(L\) is a finite abelian group \(L^\sharp\) defined as the quotient of the dual lattice~\(L^\vee\) by~\(L\) itself. The order of \(L^\sharp\) is equal to the discriminant of any Gram matrix of~\(L\). The discriminant group inherits from \(L\) a symmetric bilinear form, or, if \(L\) is even, a quadratic form. 
In this note, we provide an algorithm which computes generators of the orthogonal group \(O(L^\sharp)\), as well as a formula for its order in \Cref{thm:discr-order}.

Our motivation is the following.
Let \(L\) be an even unimodular lattice unique in its genus, and let \(M \subseteq L\) be any primitive sublattice. Denote its orthogonal complement by \(N = M^\perp\). 
By results of Nikulin \cite[Prop.~1.5.1]{nikulin:bilinear}, the group \(L/(M\oplus N)\) is the graph of an anti-isometry \(\gamma\colon M^\sharp \to N^\sharp\). The classes of primitive sublattices of \(L\) up to \(O(L)\) which are isometric to \(M\) with orthogonal complement isometric to \(N\) are parametrized by the orbit space \((O(M) \times O(N)) \backslash \gamma O(M^\sharp)\) where we let \(O(N)\) act on \(\gamma O(M^\sharp) = \{\gamma \circ g \mid g \in O(M^\sharp)\}\) via the natural map \(O(N)\to O(N^\sharp)\) and likewise for \(O(M)\).
The classes of primitive embeddings of \(M\) into \(L\) are parametrized by \(O(N) \backslash \gamma O(M^\sharp)\) \cite[Cor.~1.15.2]{nikulin:bilinear}. 

An important instance of this problem is the computation of representatives for the conjugacy classes of involutions in \(O(L)\). Here, we require \(M\) to be \(2\)-elementary, the involution acting as \(\id\) on \(M\) and as \(-{\id}\) on \(N\).

Having generators for \(O(M^\sharp)\) and its order makes these double cosets computationally accessible.  Indeed, while \(O(M)\) may be infinite, its image in \(O(M^\sharp)\) is finite and can be computed with the help of the strong approximation theorem going back to Miranda and Morrison. See \cite{brandhorst-hofmann} for an account.

We illustrate the results that can be obtained by our method starting with two known examples. The computations, which also yield explicit matrices, are carried out in the computer algebra system OSCAR \cite{OSCAR}.

\begin{example}
The even unimodular lattice of signature \((0,8)\) is the \(E_8\) lattice. There are \(9\) conjugacy classes of involutions in \(O(E_{8})\).

The even unimodular lattice of signature \((1,9)\) is the lattice \(E_{10}\). There are \(29\) conjugacy classes of involutions in \(O(E_{10})\).
\end{example}

\begin{example}
For a less trivial example, let \(M=A_2(4) \oplus A_1(4)\). As a group, \(M^\sharp\) is isomorphic to \((\ZZ/4\ZZ )^2 \times (\ZZ/8\ZZ ) \times (\ZZ/3\ZZ)\). By \Cref{thm:order}, its order is equal to \(\# O(M^\sharp) = 1536\).

Nikulin's theory shows that \(M\) embeds primitively in \(E_8\) and that the genus of its orthogonal complement is \(\mathrm{II}_{(0,5)} 8^{-2} 16^1_7 3^1\). One checks that this genus consists of \(4\) classes \(N_1, \dots, N_4\).  
For \(i \in \{1,2,3,4\}\), the order of \(O(N_i)\) is \(8,8,24,24\) and the order of its image in \(O(M^\sharp)\) is \(4,8,24,24\), respectively. Therefore, there are a total of \(1536(1/4+1/8+1/24+1/24)=704\) primitive embeddings of \(M\) into \(E_8\) up to the action of \(O(E_8)\). 

To compute the number of primitive sublattices of \(E_8\) isometric to \(M\) up to \(O(E_8)\), we compute the glue map and the double cosets and we get a total of \(2+1+1+1=5\) primitive sublattices. 
\end{example}

\subsection*{Outline}
By tensorization, we reduce ourselves to the study of lattices over the ring of \(p\)-adic numbers \(\IZ_p\). For a \(\IZ_p\)-lattice \(L\) and a positive integer \(n\), we provide generators of the group \(O(L/p^nL)\), defined as the image of the natural map \(O(L) \to \GL(L/p^nL)\). 
The study of \(O(L/p^nL)\) is further reduced to the study of \(p\)-elementary quadratic and symmetric bilinear forms. 

In \Cref{sec:2-elementary}, after recalling some general definitions, we fill a small gap in the literature (surely known to experts) by computing the order of the orthogonal groups of \(2\)-elementary quadratic and symmetric bilinear forms (\Cref{prop:order_orthogonal_elementary}). 


{In \Cref{sec:order}, we compute the order of the group \(O(L/p^nL)\) in \Cref{thm:order}, a result originally due to Watson \cite{watson1976} (\Cref{rmk:p-adic_density}) who proved it to compute the \(p\)-adic density of a lattice.

In \Cref{sec:hensel}, which constitutes the core of our paper, we give an algorithm that lifts elements of \(O(L/p^aL)\) to \(O(L/p^bL)\) for arbitrary integers \(a \leq b\). 
The key concept that we define is the notion of \(a\)-approximate triple of matrices in \(\IZ_p^{r \times r}\) (\Cref{def:approximate}).

Finally, in \Cref{sec:generators}, we provide a formula for the order of \(O(L^\sharp)\) in \Cref{thm:discr-order}. Moreover, we explain in detail how to write down generators of \(O(L/p^nL)\) for an arbitrary integer \(n\). 

The algorithms have been implemented in OSCAR \cite{OSCAR}. 

\subsection*{Acknowledgments}
We thank Ljudmila Kamenova, Giovanni Mongardi and Alex\-ei Oblomkov for organizing the workshop ``Hyperkähler quotients, singularities, and quivers'' (January 30--February 3, 2023) at the Simons Center for Geometry and Physics in Stony Brook, NY, which gave us the opportunity to finish this paper.

\section{Orthogonal groups of 2-elementary forms} \label{sec:2-elementary}

Let \(R\) be a commutative ring with identity. Given an \(R\)-module \(F\), an \emph{\(F\)-valued symmetric bilinear form} over \(R\) is a pair \((L,b)\), where \(L\) is an \(R\)-module, and \(b \colon L \times L \rightarrow F\) is a function which is \(R\)-linear in each variable and satisfies \(b(x,y) = b(y,x)\) for all \(x,y \in L\). The form is \emph{non-degenerate} if for every \(x \in L\) with \(x \neq 0\) there exist \(y \in L\) with \(b(x,y) \neq 0\).

An \(F\)-valued \emph{quadratic form} over \(R\) is a pair \((L,q)\), where \(L\) is an \(R\)-module, and \(q \colon L \rightarrow F\) is a function which satisfies \(q(rx) = r^2q(x)\) for all \(r \in R\), \(x \in L\) and for which there exists an \(F\)-valued symmetric bilinear form~\(b\) on \(L\) such that \[q(x+y) = q(x) + q(y) + b(x,y)\] for all \(x,y \in L\). 
The form \(q\) is \emph{non-degenerate} if \(b\) is.

The \emph{orthogonal group} \(O(q)\) (resp. \(O(b)\)) is the group of \(R\)-linear automorphisms \(f \colon L \rightarrow L\) such that \(q(f(x)) = q(x)\) for all \(x \in L\) (resp. \(b(f(x),f(y)) = b(x,y)\) for all \(x,y \in L\)). 
A \emph{\(p\)-elementary} quadratic (resp. symmetric bilinear) form is a non-degenerate \(\QQ/p\ZZ\)-valued quadratic (resp. symmetric bilinear) form on the \(\ZZ/p\ZZ\)-module \((\ZZ/p\ZZ)^a\). 

The aim of this section, which we achieve in \Cref{prop:order_orthogonal_elementary}, is to compute the order of the orthogonal groups of \(2\)-elementary quadratic and bilinear forms. 

\begin{example} 
We denote by \(\bu\) and \(\bv\) the \(2\)-elementary quadratic forms on \((\IZ/2\IZ)^2\) with generators \(x,y\) defined respectively by the formulas
\(q(ax+by) = ab\) and \(q(ax+by) = a^2 + ab + b^2\). We denote by \(\bar \bu\) and \(\bar \bv\) the respective bilinear forms.
We say that a \(2\)-elementary quadratic form \(q\) is \emph{even} if its associated bilinear form is alternating or, equivalently, if 
\(q \cong \bu^{\oplus m}\) or \(q \cong \bu^{\oplus m-1}\oplus \bv\).
\end{example} 

\begin{example} 
The forms \(\bw^1\) and \(\bw^3\) are the \(2\)-elementary quadratic forms on \(\IZ/2\IZ\) with generator \(x\) defined respectively by the formulas \(q(x) = \frac 12\) and \(q(x) = \frac 32\).
Their corresponding bilinear forms are isomorphic and denoted by \(\bar \bw\).
\end{example}

\begin{example} 
For any (odd) prime \(p\), we denote by \(\bh\) the \(p\)-elementary quadratic form on \((\ZZ/p\ZZ)^2\) with generators \(x,y\) defined by \(q(ax+by)=ab\). We say that a \(p\)-elementary quadratic form \(q\) is \emph{hyperbolic} if \(q\cong \bh^{\oplus m}\) for some \(m\).
\end{example}

\begin{remark} \label{rmk:forms_F2}
If \((V,b)\) is a \(2\)-elementary symmetric bilinear form, then \(V\) is a vector space over the field~\(\FF_2\) with \(2\) elements.
Since \(2b(x,y) = b(2x,y) = b(0,y) = 0\), the form~\(b\) takes values in \(\IZ/2\IZ\) and, therefore, can be considered as a symmetric bilinear form over \(\FF_2\). 
The same is true for even \(2\)-elementary quadratic forms.
Similarly, \(p\)-elementary quadratic (or bilinear) forms for \(p\neq 2\) may be seen as forms over \(\FF_p\), because they take values in \(\IZ/p\IZ\).
\end{remark}

\begin{definition} \label{def:transvection}
Let \((V,q)\) be a \(2\)-elementary quadratic form with associated bilinear form~\(b\). 
Given \(v \in V\) with \(b(v,v)=0\), we denote by \(T_v \colon V \rightarrow V\) the \emph{transvection} along~\(v\) defined by \(T_v(x) =  x + b(x,v)v\), where \(b(x,v)\) is viewed as an element of~\(\IZ/2\IZ\) (cf. \Cref{rmk:forms_F2}).
It is straightforward to check that \(T_v \in O(b)\). 
%
\end{definition}

\begin{proposition} \label{prop:order_orthogonal_elementary}
For a \(2\)-elementary quadratic form~\(q\), the orders of \(O(q)\) and \(q^{-1}(\{0\})\) are given in \Cref{table:orders_quadratic}. 
For a \(2\)-elementary symmetric bilinear form~\(b\), the order of \(O(b)\) is given in \Cref{table:orders_bilinear}.  
For a \(p\)-elementary quadratic form~\(q\) with \(p \neq 2\), the order of \(O(q)\) is given in \Cref{table:orders_quadratic_odd}.
\end{proposition}
\begin{proof}
The \(p\)-elementary case for \(p\neq 2\) is classic, see e.g. \cite[(13.3)]{kneser:quadratische_formen} for the formulas.

Let \(q\) be a \(2\)-elementary quadratic form with associated symmetric bilinear form~\(b\). 
Define \(K = \{x \mid b(x,x) = 0 \}\), which is the subspace where \(b\) is alternating. 
Clearly, \(K\), as well as \(K^\perp\) and \(K \cap K^\perp\), are invariant subspaces for both \(O(q)\) and \(O(b)\). 
By the results in \cite[\S 1.8]{nikulin:bilinear} (see also \cite[Chapter IV, §4]{miranda_morrison:embeddings}), we can write \(q \cong q' \oplus q''\), where \(q'\) is an even quadratic form and \[q'' \in \{0, \bw^1, \bw^3, \bw^1 \oplus \bw^1, \bw^3 \oplus \bw^3, \bw^1 \oplus \bw^3\}.\]

If \(q'' = 0\), i.e., if \(q\) is even, then \(q\) can be seen as a quadratic form over \(\FF_2\) (\Cref{rmk:forms_F2}). In particular, the results of \cite[(13.3)]{kneser:quadratische_formen} apply. 

If \(q'' = \bw^\varepsilon\) with \(\varepsilon \in \{1,3\}\), then \(K = q'\), \(K^\perp = q''\) and \(K \cap K^\perp = \{0\}\), so \(O(q) \cong O(q') \times O(q'') \cong O(q')\). 

Suppose now that \(q'' = \bw^\varepsilon \oplus \bw^\varepsilon\) with \(\varepsilon \in \{1,3\}\), and let \(y_1,y_2\) be the generators of the two copies of \(\bw^\varepsilon\). Then, \(K = q' \oplus \langle z \rangle\) and \(K \cap K^\perp = \langle z \rangle\), with \(z = y_1+y_2\). 
Let \(K' = K/(K\cap K^\perp)\).
Restricting to the quotient \(K'\) gives rise to a homomorphism \(O(q) \rightarrow O(b')\), where \(b'\) is the symmetric bilinear form induced by \(b\) on \(K'\). 
(Note that \(q\) does not descend to the quotient.)
We claim that this homomorphism is surjective and that its kernel has two elements, which yields \(\#O(q) = 2\#O(b')\). 
%

Indeed, the orthogonal group \(O(b')\) is in fact a symplectic group and generated by the transvections \(T_s(x)=x+b'(x,s) s\), where \(s\in K'\) \cite[2.1.11]{omeara:symplectic-groups}.
Since \(q(z) = 1\), we can find \(v \in K\) with \(\bar v = s\) and \(q(v)=1\). The reflection \(R_v(x) = x+b(x,v) v\) preserves \(q\) and restricts to \(T_s\). 
Therefore, the homomorphism is surjective. The kernel is generated by the reflection \(R_{z}\), which exchanges \(y_1\) and~\(y_2\).

Suppose now that \(q'' = \bw^1 \oplus \bw^3\) and let \(y_1,y_2\) be the generators of the copies of \(\bw^1\) and \(\bw^3\). As in the previous case, \(K = q' \oplus \langle z \rangle\) and \(K \cap K^\perp = \langle z \rangle\), with \(z = y_1+y_2\). Since now \(q(z) = 0\), restricting to \(K/(K \cap K^\perp) \cong q'\) gives rise to a homomorphism \(O(q) \rightarrow O(q')\), which is surjective. 
%
Let \(f\) be in the kernel. Then, \(f(z) = z\) and \(f(y_1)= a + \gamma z + \delta y_1\) for some \(a \in q'\) and \(\gamma,\delta \in \FF_2\). 
The condition \(q(y_1)=q(f(y_1))\) implies \(\delta =1\) and \(\gamma = q(a)\). 
Moreover, for each \(s \in q'\), we have \(f(s) = s + \alpha z\) some \(\alpha \in \FF_2\). 
The condition \(b(f(s),f(y_1)) = b(s,y_1) = 0\) gives \(\alpha = b(a,s)\). 
This shows that \(f\) is a so called Eichler transformation 
\[E_a(x) = x+b(x,z)a-b(x,a)z-q(a)b(x,z)z.\]
Eichler transformations are isometries and there are \(2^{\dim q'}\) choices for \(a\).

The formulas for \(\# q^{-1}(\{0\})\) are given in \cite[(13.6)]{kneser:quadratische_formen} for the first \(3\) cases. The others are left to the reader.

Let \(b\) be a \(2\)-elementary symmetric bilinear form. We define \(K\) as before. 
By \Cref{rmk:forms_F2}, the results in \cite[(1.20)]{kneser:quadratische_formen} and \cite[Chapter IV, §4]{miranda_morrison:embeddings} apply, so we can write \(b \cong b' \oplus b''\), with \(b' = \bar \bu^{\oplus m}\) and \(b'' \in \{0, \bar \bw, \bar \bw^{\oplus 2}\}.\)

If \(b'' = 0\), then the results of \cite[Chapt. 2]{atlas} apply. 

If \(b'' = \bar \bw\), then \(K = b'\), \(K^\perp = b''\) and \(K\cap K^\perp = \{0\}\), so \(O(b) \cong O(b') \times O(b'') \cong O(b')\).

Finally, suppose that \(b'' = \bar \bw^{\oplus 2}\) and let \(x_1,\dots,x_n\) be a basis of \(b'\) and \(y_1,y_2\) be the generators of the two copies of \(\bar \bw\). Then, \(K = b' \oplus \langle z \rangle\) and \(K \cap K^\perp = \langle z \rangle\), with \(z = y_1+y_2\). Restricting to \(K' = K/(K \cap K^\perp)\) gives rise to a homomorphism \(O(b) \rightarrow O(b|_{K'})\cong O(b')\), which is clearly surjective. 
Let \(f\) be an element of the kernel. We have \(f(x_i) = x_i + \alpha_i z\), \(f(z)=z\) and \(f(y_1)=x + \beta z + \gamma y_1\) for scalars \(\alpha_1,\dots, \alpha_n, \beta, \gamma \in \FF_2\) and \(x \in \langle x_1,\dots, x_n \rangle\).
As before, it holds \(\gamma=1\). Moreover, \(x\) is determined by the linear equations \(b(x+y_1,f(x_i))=0\) for \(i\in \{1,\dots,n\}\).
We calculate that \(b(f(y_1,y_1))=b(x,x)+b(z,z)+b(y_1,y_1)= 0 + 0 + 1\) is independent of \(\beta\). 
Therefore, any choice of \(\alpha_1,\dots, \alpha_n\) and \(\beta\) yields an element of the kernel. Its order is \(2^{n+1}\) as claimed. 
\end{proof}

\begin{table}
    \caption{Orthogonal groups of \(2\)-elementary quadratic forms, where \(q'\) is even and \(\varepsilon \in \{1,3\}\).}
    \label{table:orders_quadratic}
    \renewcommand*{\arraystretch}{1.5}
    \begin{tabular}{lll}
        \toprule
        \(q\) & \(\#O(q)\) & \(\# q^{-1}(\{0\})\) \\
        \midrule
        \(\bu^{\oplus m}\) & \(2^{m(m-1)+1}(2^m - 1)\prod_{k=1}^{m-1}\left(2^{2k}-1\right)\) & \(2^{m-1}(2^m+1)\)\\
        \(\bu^{\oplus m-1} \oplus \bv_{1}\) & \(2^{m(m-1)+1}(2^m + 1)\prod_{k=1}^{m-1}\left(2^{2k}-1\right)\) &\(2^{m-1}(2^m -1)\)\\
        \(q' \oplus \bw^{\varepsilon}\) & \(\# O(q')\)& \(\# q'^{-1}(0)\)\\
        \(q' \oplus \bw^{\varepsilon}\oplus \bw^{\varepsilon}\) & \(2 \# O(b_{q'})\) & \( 2^{\dim q'} \)\\
        \(q' \oplus \bw^{1} \oplus \bw^{3}\) & \(2^{\dim q'}\# O(q') \) & \(2\# q'^{-1}(0)\)\\
        \bottomrule
    \end{tabular}
\end{table}

\begin{table}
    \caption{Orthogonal groups of \(2\)-elementary symmetric bilinear forms.}
    \label{table:orders_bilinear}
    \renewcommand*{\arraystretch}{1.5}
    \begin{tabular}{llll}
        \toprule 
        \(b\) & \(\#O(b)\)  \\
        \midrule
        \(\bar \bu^{\oplus m}\) &  \(2^{m^2}\prod_{k=1}^{m}\left(2^{2k}-1\right)\) & \\
        \(\bar \bu^{\oplus m} \oplus \bar \bw\) & \(\#   O(\bar \bu^{\oplus m})\) \\
        \(\bar \bu^{\oplus m} \oplus \bar \bw^{\oplus 2}\) &  \(2^{2m+1} \#  O(\bar \bu^{\oplus m})\) \\
        \bottomrule 
    \end{tabular}
\end{table}

\begin{table}
\caption{Orthogonal groups of \(p\)-elementary quadratic forms, \(p \neq 2\). }
\label{table:orders_quadratic_odd}
    \renewcommand*{\arraystretch}{1.5}
    \begin{tabular}{lll}
        \toprule 
        \(q\) &\(\dim q\)& \(\#O(q)\)  \\
        \midrule
hyperbolic & \(2m\)&\( 2p^{m(m-1)}(p^m - 1)\prod_{k=1}^{m-1}\left(p^{2k}-1\right)\)\\
 else & \(2m\)&\( 2p^{m(m-1)}(p^m + 1)\prod_{k=1}^{m-1}\left(p^{2k}-1\right)\)\\
         & \(2m+1\) &\(2p^{m^2}\prod_{k=1}^{m}\left(p^{2k}-1\right)\)\\
        \bottomrule 
    \end{tabular}
\end{table}%

\begin{remark} \label{rmk:[O(b):O(q)]}
By inspection of \autoref{table:orders_quadratic} and \autoref{table:orders_bilinear}, we see that if \(q\) is a \(2\)-elementary quadratic form and \(b\) its associated bilinear form, then \[[O(b):O(q)]=\# q^{-1}(\{0\}).\]
\end{remark}

\begin{definition} \label{def:defect}
Let \((V,q)\) be a \(2\)-elementary quadratic form with associated bilinear form~\(b\).
The \emph{defect} of \(f \in O(b)\), is defined as the unique element \(v_f \in V\) satisfying the following condition for all \(x \in V\):
\[
    b(v_f,x) = q(fx) - q(x).
\]
Indeed, it holds \(2q(fx) = b(fx,fx) = b(x,x) = 2q(x)\), so the map \(x \mapsto q(fx)-q(x)\) takes values in \(\IZ/2\IZ\) and is linear, because \(q(f(x+y)) - q(x+y) = q(fx)-q(x) + q(fy)-q(y)\).
Hence, \(v_f\) is the solution of a linear system over \(\FF_2\).
\end{definition}

\begin{corollary} \label{cor:bijection}
If \(q\) is a \(2\)-elementary quadratic form with associated bilinear form~\(b\), then the maps (see \Cref{def:transvection,def:defect})
\[
    \Phi \colon q^{-1}(\{0\}) \longrightarrow O(q) \backslash O(b),  \qquad v \longmapsto O(q)T_v
\]
and
\[
    \Psi \colon O(q) \backslash O(b) \longrightarrow q^{-1}(\{0\}), \qquad O(q)f \longmapsto v_f
\]
are mutually inverse. In particular, it holds \(q(v_f) = q(fv_f) = 0\) for any \(f \in O(b)\). 
\end{corollary}
\begin{proof}
Any \(v \in q^{-1}(\{0\})\) satisfies \(b(v,v) = 0\), so \(\Phi\) is well defined. 
Given \(g \in O(q)\), we have \(b(v_{gf},x) = q(gfx)-q(x) = q(fx)-q(x) = b(v_{f},x)\),
hence \(v_{g f} = v_f\). 
In particular, the function
\(\tilde{\Psi}\colon O(q) \backslash O(b) \rightarrow V\) mapping \(O(q)f\) to \(v_f\)
is well defined. Note that \(q(v_f)=0\) is yet to be proven.

Let \(v \in q^{-1}(\{0\})\). Then \(\tilde{\Psi} \circ \Phi (v) = \tilde{\Psi}(T_v)\) and for \(x\in V\) we have
\begin{align*}
    b(\tilde{\Psi}(T_v),x) & = q(T_v(x))-q(x) = q(x + b(x,v)v)-q(x) \\
        & q(x) + b(x,v)^2 + b(x,v)^2q(v)-q(x) = b(x,v)^2 = b(x,v).
\end{align*}
Hence, \(v=\tilde{\Psi}(T_v)\) as desired, i.e., \(\tilde{\Psi}\) is left inverse to \(\Phi\). Since both sets are finite and have the same cardinality by \Cref{rmk:[O(b):O(q)]}, \(\Phi\) is bijective.
Therefore, the image of \(\tilde{\Psi}\) is indeed \(q^{-1}(\{0\})\), and \(\Psi\) is the inverse of \(\Phi\).
This also proves \(q(v_f) = 0\), which in turn implies \(q(fv_f) = q(v_f) + b(v_f,v_f)= 3q(v_f) = 0\).
\end{proof}

\section{The order formula} \label{sec:order}
Let \(R\) be a principal ideal domain with field of fractions \(Q\). 
A \emph{lattice} over~\(R\) is a \(Q\)-valued quadratic form~\((L,q)\) over~\(R\) on a finitely generated, free \(R\)-module.
The \emph{dual lattice} of \(L\) is denoted by \(L^\vee = \{ x \in L\otimes Q \mid b(x,L)\subseteq R\}\cong \Hom_R(L,R)\).
Following Kneser \cite[(26.1)]{kneser:quadratische_formen}, a lattice \(L\) is called \emph{integral} if \(b(L,L) \subseteq R\); 
\emph{even} if \(b(x,x) \in 2R\) for all \(x \in L\) (equivalently, if \(q(L) \subseteq R\)); 
\emph{odd} if \(L\) is integral, but not even; 
\emph{unimodular} if \(L\) is integral and \(\det (L) \in R^\times/R^{\times 2}\); 
\emph{modular} or, more precisely, \emph{\(p^i\)-modular}  if \((L,p^{-i}q)\) is unimodular for \(p \in R\), \(i \in \IZ\).

Here, we take \(R = \IZ_p\), the ring of \(p\)-adic integers. 
An orthogonal decomposition of a \(\IZ_p\)-lattice \((L,q) = \bigoplus_i (L_i, p^i q_i)\) with unimodular lattices \((L_i, q_i)\) is called a \emph{Jordan decomposition}. 
The summands \((L_i, p^i q_i)\) are called \emph{Jordan constituents}.  
Every \(\IZ_p\)-lattice admits a Jordan decomposition, not necessarily unique (see, e.g., \cite[Ch.~15, \S 7]{conway_sloane:sphere_packings}). 

The aim of this section, which we achieve in \Cref{thm:order}, is to determine the order of the image of the natural group homomorphism \(O(L) \rightarrow \GL(L/p^nL)\), where \(\GL(L/p^nL)\) denotes the group of linear automorphisms of \(L/p^nL\) as a \(\IZ/p^n\IZ\)-module.

\begin{definition} \label{def:oddity_vector}
Given a unimodular \(\IZ_2\)-lattice \((E,q)\), an element \(v_E \in E\) with the property that
\[
    b(v_E, x) \equiv b(x,x) \mod 2
\]
for all \(x \in E\) is called an \emph{oddity vector} of \(E\).
If \(v_E'\) is another oddity vector, then \(v_E'=v_E+2y\) for some \(y \in E\) and \(q(v_E)\equiv q(v_E') \bmod 4\).
Conway and Sloane \cite{conway_sloane:sphere_packings} call the quantity \(b(v_E,v_E) \bmod 8\) the \emph{oddity} of \(E\), whence the name of \(v_E\). 
We note that the oddity vector is also known as the characteristic vector of \(E\).
By definition, the oddity vector and the oddity of \((E, 2^i q)\) are those of \((E, q)\).
\end{definition}

\begin{lemma} \label{lem:lift_embedding_mod4}
Let \(E\) and \(G\) be integral \(\IZ_2\)-lattices of ranks \(e\) and \(g\), respectively. Suppose that \(E\) is unimodular, and that the bilinear form \(b_G\otimes \FF_2\) has rank \(e\), i.e., any Jordan constituent \(G_0\) is also of rank \(e\). If \(f \colon E \rightarrow G\) is a linear map, then there is an isometric embedding \(f' \colon E \rightarrow G\) with
\[
    f' \equiv f \mod 2
\]
if and only if the following conditions hold:
\begin{align}
    q(fx)    & \equiv q(x)    \mod 2 \quad \text{ for all \(x \in E\)}, \label{eq:mod2_1} \\
    q(fv_E)  & \equiv q(v_E)  \mod 4.                 \label{eq:mod2_2}
\end{align}
If this is the case, then for each \(n \geq 2\), the number of possibilities for \(f'\) modulo \(2^n\) is given by \(2^k\) with
\[
    k = (n-1)eg - (n-1)\frac{e(e+1)}{2} + t,
\]
where we set \(t = 1\) if \(E\) is odd, and \(t = 0\) if \(E\) is even.
\end{lemma}
\begin{proof}
\Cref{eq:mod2_1} depends only on \(f\) modulo \(2\), so its necessity is clear. 
Moreover, it implies that \(f\) induces an isometry of bilinear \(\FF_2\)-modules 
\(E/2E \to G/(G\cap 2G^\vee)\) as both \(E = E_0\) and~\(G_0\) have rank \(e\). This implies that \(b(fv_E,y)\equiv b(y,y) \bmod 2\) for all \(y \in G\).
We have therefore \[
    q(fv_E + 2y) - q(fv_E) = 2b(fv_E,y) +4q(y) \equiv 8q(y)\equiv 0 \mod 4,
\]
where we used that \(G\) is integral, that is, \(2q(G) \subseteq \IZ_2\), in the last step.
This shows that \cref{eq:mod2_2} also depends only on \(f\) modulo \(2\).

For the sufficiency, let \(f\) satisfy \cref{eq:mod2_1,eq:mod2_2}. We want to find a linear map \(h \colon E \rightarrow G\) such that \(f' = f + 2h\) is an isometric embedding. 
Recall that \(\IZ_2\) is complete and \(4q(G) \subseteq 2\IZ_2\). 
We denote by \(b_f(y)\) the map \(x \mapsto b(fx,y)\).

By \cite[Satz (15.5)]{kneser:quadratische_formen}, we only need \(f'\) to satisfy the following conditions:
\[
    E^\vee = b_{f'}(G) + 2E^\vee, \qquad q(f'x) \equiv q(x) \mod 4.
\]
The former condition is satisfied for any choice of \(h\). Indeed, \(E = E^\vee\) by assumption, and \(b_{f'}(G) \subseteq E^\vee\) because \(G\) is integral. Moreover, it holds for all \(e,x \in E\)
\[
    b(x,e) \equiv b(fx,fe) \equiv b(f'x,f'e) \equiv b_{f'}(f'e)(x) \mod 2,
\] 
which shows \(E \subseteq b_{f'}(G) + 2E^\vee\). We hence turn to the latter condition.

By \cref{eq:mod2_2}, we can find (see \cite[(2.3)]{kneser:quadratische_formen}) a not necessarily symmetric, bilinear form \(a \colon E \times E \rightarrow \IZ_2\) such that for all \(x \in E\)
\[
    q(fx) - q(x) = 2a(x,x), \qquad a(v_E,x) \equiv 0 \mod 2.
\]
Since \(fv_E\) is an oddity vector for \(G\), we have
\[
    b(fv_E,hx) \equiv 2q(hx) \mod 2,
\]
which in turn gives
\begin{align*} 
    q(f'x) & = q(fx) + 2b(fx,hx) + 4q(hx) \\
           & \equiv q(x) + 2\big(a(x+v_E,x) + b(fx+fv_E,hx)\big) \mod 4.
\end{align*}
Therefore, \(f'\) preserves the quadratic forms modulo \(4\) if and only if \(h\) satisfies
\begin{equation}\label{eqn:h}
b(f(x+v_E),hx) \equiv -a(x+v_E,x) \mod 2
\end{equation}
for all \(x \in E\).
Since \(E^\vee = b_f(G) + 2E^\vee\), we can find for every \(x \in E\) an element \(hx \in G\) satisfying \(b_f(hx) \equiv a(\cdot,x) \bmod 2E^\vee\). We do so on a basis of \(E\) and we define \(h\) by extending linearly. Unraveling the definitions, we obtain \(b(f(x+v_E))=b_f(hx)(x+v_E) \equiv a(x+v_E,x) \bmod 2\), i.e., \cref{eqn:h} holds.

Finally, we count the possibilities for \(f' = f + 2h\) modulo \(2^n\) starting with \(n = 2\). We need to compute the dimension of the homogeneous solution space in \(h\) for the equation
\begin{equation} \label{eq:lifting_h_mod2}
    b(fx + fv_E, hx) \equiv 0 \mod 2 \quad \text{for all \(x \in E\)}.
\end{equation} 
We choose a basis in \(G\) composed of \(e\) elements in \(f(E)\) and \(g-e\) elements in \(f(E)^\perp\). 
Let \(F,H \in \FF_2^{e\times g}\) be the matrices with coefficients in \(\IZ_2/2\IZ_2 \cong \FF_2\) representing \(f\) and \(h\) modulo \(2\), respectively. We must determine the \(eg\) coefficients of \(H\). Note that the \(e(g-e)\) coefficients in the last \(g-e\) columns can be chosen freely.

Denoting the Gram matrix of \(G\) by the same symbol \(G\), and by \(x_0 \in \FF_2^e\) the vector corresponding to \(v_E\), \cref{eq:lifting_h_mod2} translates into the following matrix equation for \(\tilde{H} = FGH^\intercal \):
\[
    x\tilde{H}x^\intercal   + x_0\tilde{H}x^\intercal  = 0 \quad \text{for all \(x \in \FF_2^e\)}.
\]
Up to coordinate change, we can assume \(x_0 = te_1\), with \(t = 0\) if \(E\) is even or \(t = 1\) if \(E\) is odd, and \(e_1\) is the first standard basis vector. By polarization, we obtain the following equivalent system of equations:
\[
    \tilde{H}_{ij} + \tilde{H}_{ji} = 0, \quad \tilde{H}_{i,i} + t \tilde{H}_{1i} = 0 \quad \text{for all \(i,j \in \{1,\ldots,e\}\)}.
\]
Hence, we can choose \(e(e-1)/2 + t\) coefficients freely. Indeed, if \(t = 1\), then the equation \(\tilde{H}_{i,i} + t \tilde{H}_{1i} = 0\) is trivial for \(i = 1\). In total, we get \(2^k\) possibilities for \(h\) modulo~\(2\), that is, for \(f'\) modulo~\(4\), with
\[
    k = e(g-e) + \frac {e(e-1)}2  + t = eg - \frac{e(e+1)}{2} + t
\]

Now, every \(f'\) modulo \(2^{n-1}\) can be extended modulo~\(2^n\) in exactly \(2^{eg - e(e+1)/2}\) ways for every \(n \geq 3\) by \cite[Zusatz (15.4)]{kneser:quadratische_formen}. This finishes the proof.
\end{proof}

\begin{definition} \label{def:qi}
Let \((L,q) = \bigoplus_i (L_i, p^i q_i)\) be a Jordan decomposition of a \(\ZZ_p\)-lattice. Using the isomorphism \(\QQ_p/p\ZZ_p \cong \QQ/p\ZZ\), we define on \(L_i/pL_i\) the \(p\)-elementary quadratic form \[\bar{q}_i(x + pL_i) = q_i(x) \bmod p\] and its associated bilinear form \[\bar{b}_i(x + pL_i,y + pL_i) = b_i(x,y) \bmod p.\]
\end{definition}

\begin{lemma} \label{lem:liftable_embedding_mod2}
Let \((L,q) = \bigoplus_{i \geq 0} (L_i, p^i q_i)\) be a Jordan decomposition of an integral \(\IZ_2\)-lattice \(L\) of rank \(r\), with \(L_i\) of rank \(r_i\). Let \(t_i = 0\) if \((L_i, q_i)\) is even, and \(t_i = 1\) else. Set \(t = (t_0,t_1,t_2)\).
Then, the number of embeddings \(\bar f \in \Hom(L_0,L) \otimes \FF_2\) modulo~\(2\) which lift to isometric embeddings \(f \in \Hom(L_0,L)\) is given by
\[
    2^{r_0(r-r_0) - c_0} \cdot
    \begin{cases}
        \#O(\bar{q}_0) & \text{for \(t_1 = 0\)}, \\
        \#O(\bar{b}_0) & \text{for \(t_1 = 1\)},
    \end{cases}
\]
with
\[
    c_0 =
    \begin{cases}
        0                                            & \text{for \(t = (0,0,0)\)}, \\
        0                                           & \text{for \(t = (0,0,1)\)}, \\
        r_0                                         & \text{for \(t = (0,1,0)\)}, \\
        r_0                                          & \text{for \(t = (0,1,1)\)}, \\
        r_1 - \log_2(\#\bar{q}_1^{-1}(0))           & \text{for \(t = (1,0,0)\)}, \\
        1                                             & \text{for \(t = (1,0,1)\)}, \\
        r_0 + r_1 - 1 - \log_2(\#\bar{q}_1^{-1}(0)) & \text{for \(t = (1,1,0)\)}, \\
        r_0 + 1                                     & \text{for \(t = (1,1,1)\)}.
    \end{cases}
\]
\end{lemma}
\begin{proof}
Let \(\bar f = \bigoplus \bar f_i \in \Hom(L_0,L) \otimes \FF_2\) be an isometric embedding modulo \(2\) that lifts to an isometric embedding \(f = \bigoplus f_i \in \Hom(L_0,L)\), where \(\bar f_i \in \Hom(L_0,L_i) \otimes \FF_2\) denotes the image of \(f_i \in \Hom(L_0,L_i)\). Note that the term \(2^{r_0(r-r_0)}\) is the number of possibilities for \(\bigoplus_{i \geq 1} \bar f_i\) without any constraints.

Since \(b(L_i,L_i) \subseteq 2^i\IZ_2\) for every \(i\), we have for all \(x,y \in L_0\) that 
\[
    b(x, y) = b(fx, fy) \equiv b(f_0x, f_0y) \mod 2,
\]
which is equivalent to \(\bar f_0 \in O(\bar{b}_0)\). Moreover, it always holds
\begin{equation} \label{eq:t1=1}
    q(x) = q(fx) \equiv q(f_0x) + q(f_1x) \mod 2,
\end{equation}
because \(q(L_i) \subseteq 2^{i-1}\IZ_2\) for every \(j\).
Now, if \(t_1 = 0\), then \(q(L_1) \subseteq 2\IZ_2\), so \cref{eq:t1=1} simplifies to
\begin{equation} \label{eq:t1=0}
    q(x) \equiv q(f_0x) \mod 2,
\end{equation}
which is equivalent to \(\bar f_0 \in O(\bar{q}_0)\). 

We are now ready to count the constraints given by \Cref{lem:lift_embedding_mod4}. 

Suppose that \(t_0 = 0\), that is, \(v_{L_0} = 0\) (\Cref{def:oddity_vector}). If \(t_1 = 0\), then \cref{eq:mod2_2} is automatically satisfied, while \cref{eq:mod2_1} is equivalent to \cref{eq:t1=0}. Hence, any choice of \(\bar f = \bigoplus \bar f_j\) lifts as long as \(\bar f_0 \in O(\bar{q}_0)\), but without any constraints on \(\bar f_i\) for \(i \geq 1\), that is, \(c_0 = 0\).
If \(t_1 = 1\), \cref{eq:t1=1} represents \(r_0\) independent linear equations on \(\bar f_1\) for every choice of \(\bar f_0 \in O(\bar{b}_0)\). Moreover, we have no constraints on \(\bar f_i\) for \(i \geq 2\), so we set \(c_0 = r_0\) in this case.

Now, suppose that \(t_0 = 1\) and let \(v = v_{L_0} \in L_0\). Given that \(\bar f_0 \in O(\bar{b}_0)\), we have for all \(x \in L_0\)
\[
    b(f_0v,x) \equiv b(v,f_0^{-1}x) \equiv b(f_0^{-1}x,f_0^{-1}x) \equiv b(x,x),
\]
which implies that \(f_0(v) = v + 2x_0\) for some \(x_0 \in L_0\). Thus, it holds
\begin{equation} \label{eq:q(f0v)=q(v)_mod4}
    q(f_0v) \equiv q(v) + 4q(x_0) + 2b(x_0,x_0) \equiv q(v) + 8q(x_0) \equiv q(v) \mod 4.
\end{equation}
In particular, \cref{eq:mod2_2} is equivalent to
\begin{equation} \label{eq:t0=1_t2=1}
    q(f_1v) + q(f_2v) \equiv 0 \mod 4.
\end{equation}
If in addition \(t_2 = 0\), then \(q(L_2) \subseteq 4\IZ_2\), so \cref{eq:t0=1_t2=1} simplifies to
\begin{equation} \label{eq:t0=1_t2=0}
    q(f_1v) \equiv 0 \mod 4.
\end{equation}

Recall that \(\bar f_0 \in O(\bar{q}_0)\) whenever \(t_1 = 0\). For \(t = (1,0,0)\), \cref{eq:t0=1_t2=0} represents the \(r_1\) independent linear equations \(\bar f_1 \bar v=\bar w\) for \(\bar{f}_1\) for every choice of \(\bar f_0\) and \(\bar w \in \bar{q}_1^{-1}(\{0\})\). We have no other conditions on \(\bigoplus_{j \geq 2} \bar f_j\), so we put \(c_0 = r_1 - \log_2(\#\bar{q}_1^{-1}(\{0\}))\).

For \(t = (1,0,1)\), \cref{eq:t0=1_t2=1} represents only one linear condition on \(\bar f_2\). Thus, we merely put \(c_0 = 1\) in this case.

For \(t = (1,1,0)\), we obtain a contribution of \(r_1-\log_2(\#\bar{q}_1^{-1}(\{0\}))\) to \(c_0\) from \cref{eq:t0=1_t2=0} as before. 
On the other hand, \cref{eq:q(f0v)=q(v)_mod4,eq:t0=1_t2=0} imply that \cref{eq:t1=1} is automatically satisfied for \(v\). 
Thus, we obtain a contribution of \(r_0 - 1\) to \(c_0\) from \cref{eq:t1=1}.

For \(t = (1,1,1)\), we get a contribution of~\(r_0\) to \(c_0\) from \cref{eq:t1=1} and a contribution of~\(1\) from \cref{eq:t0=1_t2=1}.
\end{proof}

Given a \(\IZ_p\)-lattice \(L\), we denote the image of \(O(L) \rightarrow \GL(L/p^nL)\) by \(O(L/p^nL)\). In the next theorem, we provide a formula for the order of this group.

\begin{theorem} \label{thm:order}
Let \((L,q) = \bigoplus_{i \in \IZ} (L_i, p^i q_i)\) be a Jordan decomposition of a (not necessarily integral) \(\IZ_p\)-lattice \(L\) of rank \(r\), with \(L_i\) of rank \(r_i\). 
Given \(n \geq 1\), set
\[
    v = (n-1)\frac{r(r-1)}{2} + \sum_{i < j} r_i r_j.
\]
If \(p \neq 2\), then
\[
    \#O(L/p^nL) = p^v \prod_{i \in \IZ} \# O(\bar{q}_i).
\]
If \(p = 2\), let \(t_i = 0\) if \((L_i, q_i)\) is even, or \(t_i = 1\) else. Then,
\begin{equation} \label{eq:c}
    \#O(L/2^nL) = 2^v \prod_{i \in \IZ} 
    \begin{Cases}
    \#O(\bar{q}_i) & \text{for } t_{i-1}=t_{i+1}=0\\ 
    2^{-r_i}\#O(\bar{b}_i) & \text{else}
    \end{Cases} 2^{t_i\delta_{n\geq 2} - s_i}, 
\end{equation}
where
\[
    s_i = \begin{cases}
        1   & \text{for \((t_{i-1}, t_{i}, t_{i+1}) = (1,0,1)\)}, \\
        -1  & \text{for \((t_{i-1}, t_{i}, t_{i+1}) = (0,1,1)\)}, \\
        1   & \text{for \((t_{i-1}, t_{i}, t_{i+1}) = (1,1,1)\)}, \\
        0   & \mbox{else.}
    \end{cases}
\]
\end{theorem}
\begin{proof}
We give a proof for \(p = 2\). The proof for \(p \neq 2\) is analogous, but much simpler, because there are basically no case distinctions, so we leave it to the reader.

Since the Jordan decomposition is finite, there exists \(i_0\) and \(k\) with \((L,q) = \bigoplus_{i = i_0}^{i_0 + k} (L_i, p^i q_i)\). The proof is by induction on \(k\). After rescaling the form~\(q\), we may assume without loss of generality that \(i_0 = 0\). 

Set \(L' = \bigoplus_{i=1}^{k} L_i\) with invariants \(r_i',t_i',s_i'\).
Let \(f \colon L_0\hookrightarrow L\) be an embedding and \(L''\) be the orthogonal complement of the image. 
Then, \(L' \cong L''\) by \cite{kneser:witt_for_local_rings} (or \cite[Satz (4.3)]{kneser:quadratische_formen}), and we can extend \(f\) to an element of \(O(L)\). 
Counted modulo \(2^n\), there are \(b = \# O(L'/2^nL')\) such extensions. 
Let \(a\) be the number of isometric embeddings \(f \colon L_0 \rightarrow L\) modulo \(2^n\) which lift to an isometric embedding \(f \in \Hom(L_0,L)\). 
We have obtained \(\#O(L/2^nL) = ab\).

Note that \(a = a_1a_2\), where \(a_1\) is the number of isometric embeddings \(f\) modulo~\(2\), which is given by \Cref{lem:liftable_embedding_mod2}, and \(a_2\) is the number of possible extensions \(f'\) modulo \(2^n\) with \(f' \equiv f\) modulo \(2\). By \Cref{lem:lift_embedding_mod4}, \(a_2\) does not depend on \(f\). 

If \(k = 0\), then \(L = L_0\), \(r = r_0\), \(L' = 0\), and \(b = 1\). Furthermore, \(s_0 = 0\). We abbreviate \(\delta  = \delta_{n \geq 2}\). For \(n = 1\), \Cref{lem:liftable_embedding_mod2} gives that \(\#O(L/2L) = a_1 = \# O(\bar{q}_0)\). For \(n \geq 2\), we have     \[
        \#O(L/2^nL) = a_1 a_2 = \# O(\bar{q}_0) \cdot 2^k,
    \]
with \(k = v + t_0\delta\) given by \Cref{lem:lift_embedding_mod4}. 

Let now \(k>0\). By the induction hypothesis for \(L'\),
\begin{align*}
    b & = 2^{v'} \prod_{i=1}^{k} \begin{Cases} \#O(\bar{q}_i') & \text{\(t_{i-1}'=t_{i+1}'=0\)} \\ \#O(\bar{b}_i')2^{-r_i} & \text{else} \end{Cases} 2^{t'_i\delta-s'_i} \\
    & = 2^{v'}\begin{Cases} \#O(\bar{q}_1) & t_2 = 0 \\ \#O(\bar{b}_1) & t_2 = 1 \end{Cases} 2^{t_1\delta - s_1'} \prod_{i = 2}^k \cdots 
\end{align*}
where \(v'=(n-1)(r-r_0)(r-r_0-1)/2+\sum_{0<i<j}r_ir_j\).
By \Cref{lem:liftable_embedding_mod2,lem:lift_embedding_mod4}, we have
\[
    a = 2^{w-c_0+t_0\delta} \begin{Cases} \#O(\bar{q}_0) & t_1 = 0 \\ \#O(\bar{b}_0) & t_1 = 1 \end{Cases},
\] 
where \(w = (n-1)(r_0r - r_0(r_0+1)/2) + r_0(r-r_0)\).

We have to prove \(ab = c\), where \(c\) is given by \cref{eq:c}:
\[
    c = 2^v \begin{Cases} \#O(\bar{q}_0) & t_1 = 0 \\ \#O(\bar{b}_0)2^{-r_0} & t_1 = 1 \end{Cases} 2^{t_0\delta - s_0} \begin{Cases} \#O(\bar{q}_1) & t_0 = t_2 = 0 \\ \#O(\bar{b}_1)2^{-r_1} & \text{else} \end{Cases} 2^{t_1\delta - s_1} \prod_{i = 2}^k \cdots
\]

Given that \(t_i' = t_i\) for \(i \geq 1\), it also holds \(s_i = s_i'\) for \(i \geq 2\). Therefore, the factors appearing in the products in \(b\) and \(c\) relative to \(i \geq 2\) are equal. One can easily check that \(v'+w = v\).
We confirm that \(ab = c\) for each of the \(8\) possible \(t = (t_0,t_1,t_2)\) separately. It is a matter of bookkeeping.

For \(t=(0,0,0)\),
\begin{align*}
    ab  & = 2^{w-0+0}\#O(\bar{q}_0) \cdot 
            2^{v'} \#O(\bar{q}_1)2^{0-0} \cdots, \\
    c   & = 2^{v}\#O(\bar{q}_0)2^{0-0} \#O(\bar{q}_1)2^{0-0} \cdots.
\end{align*}

For \(t=(0,0,1)\),
\begin{align*}
    ab  & = 2^{w-0+0}\#O(\bar{q}_0) \cdot 
            2^{v'}\#O(\bar{b}_1)2^{-r_1}2^{0-0} \cdots, \\
    c   & = 2^{v}\#O(\bar{q}_0)2^{0-0} \#O(\bar{b}_1)2^{-r_1}2^{0-0} \cdots.
\end{align*}

For \(t=(0,1,0)\),
\begin{align*}
    ab  & = 2^{w-r_0+0}\#O(\bar{b}_0) \cdot
            2^{v'}\#O(\bar{q}_1)2^{\delta-0} \cdots, \\
    c   & = 2^{v}\#O(\bar{b}_0)2^{-r_0+0-0} \#O(\bar{q}_1)2^{\delta-0} \cdots.
\end{align*}

For \(t=(0,1,1)\),
\begin{align*}
    ab  & = 2^{w-r_0+0}\#O(\bar{b}_0) \cdot
            2^{v'}\#O(\bar{b}_1)2^{-r_1}2^{\delta-(-1)} \cdots, \\
    c   & = 2^{v}\#O(\bar{b}_0)2^{-r_0+0-0} \#O(\bar{b}_1)2^{-r_1}2^{\delta-(-1)} \cdots.
\end{align*}

For \(t = (1,0,0)\),
\begin{align*}
    ab  & = 2^{w-r_1+\delta}\#\bar{q}_1^{-1}(\{0\}) \# O(\bar{q}_0) \cdot
        2^{v'}\#O(\bar{q}_1)2^{0-0} \cdots, \\
    c   & = 2^v \#O(\bar{q}_0)2^{\delta-0} \#O(\bar{b}_1)2^{-r_1}2^{0-0} \cdots.
\end{align*}

For \(t = (1,0,1)\),
\begin{align*}
    ab  & = 2^{w-1+\delta}\# O(\bar{q}_0) \cdot
            2^{v'}\#O(\bar{b}_1)2^{-r_1}2^{0-0} \cdots, \\
    c   & = 2^v \#O(\bar{q}_0)2^{\delta-0} \#O(\bar{b}_1)2^{-r_1}2^{0-1} \cdots.
\end{align*}

For \(t = (1,1,0)\),
\begin{align*}
    ab  & = 2^{w-(r_0+r_1-1)+\delta}\#\bar{q}_1^{-1}(\{0\})\#O(\bar{b}_0) \cdot
            2^{v'}\#O(\bar{q}_1)2^{\delta-0} \cdots, \\
    c   & = 2^{v}\#O(\bar{b}_0)2^{-r_0}2^{\delta-(-1)} \#O(\bar{b}_1)2^{-r_1}2^{\delta+0} \cdots.
\end{align*}

For \(t = (1,1,1)\),
\begin{align*}
    ab  & = 2^{w-(r_0+1)+\delta}\#O(\bar{b}_0) \cdot
        2^{v'}\#O(\bar{b}_1)2^{-r_1}2^{\delta-(-1)} \cdots, \\
    c   & = 2^{v}\#O(\bar{b}_0)2^{-r_0}2^{\delta-(-1)} \#O(\bar{b}_1)2^{-r_1}2^{\delta-1} \cdots.
\end{align*}

Indeed, we always have \(ab = c\). For \(t = (1,0,0)\) and \(t = (1,1,0)\), we use the fact that \(\#\bar{q}_1^{-1}(\{0\}) \#O(\bar{q}_1)=\#O(\bar{b}_1)\), which we observed in \Cref{rmk:[O(b):O(q)]}.
\end{proof}

\begin{remark}
Note that the \(\bar{q}_1\) appearing in the cases \(t = (1,0,0)\) and \(t = (1,1,0)\) of \Cref{lem:liftable_embedding_mod2} may depend on the chosen Jordan decomposition of \(L\). On the other hand, all formulas appearing in \Cref{thm:order} are independent of the chosen Jordan decomposition.
\end{remark}

\begin{remark} \label{rmk:p-adic_density}
The numbers \(N(L,p^n)=\# \{X \in \IZ_p^{n \times n}| XGX^\intercal \equiv G \mod p^n \}\) were computed by Watson \cite{watson1976}.
We have
\[N(L, p^n) = \#O(L/p^nL) \prod_{i<j}p^{i r_i r_j} \prod_{i}p^{i r_i(r_i+1)/2}\]
if \(n\) is large enough. This gives the connection with the \(p\)-mass \(m_p\) as defined by Conway and Sloane \cite[p. 281]{conway-sloane:mass_formula}:
\[m_p(L) = \frac{p^{nr(r-1)/2+s(r+1)/2}}{N(L,p^n)},\]
where \(s\) is the highest power of \(p\) dividing the determinant of \(L\).
For \(p \neq 2\), we checked that our mass formulas and Conway--Sloane's are equivalent.
For \(p = 2\), though, we only confirmed that both
formulas give the same masses for a few thousand genera of small rank and determinant.
\end{remark}

\section{Hensel lifting algorithms} \label{sec:hensel}
Let \(f\colon L \to M\) be a linear map of \(\ZZ_p\)-lattices of the same rank. 
We say that \(f\) is \emph{compatible} if \(f(L \cap p^iL^\vee)=M \cap p^iM^\vee\) for all \(i \in \ZZ\).
For an integer \(a \geq 0\), we call a compatible linear map \(f\) \emph{\(a\)-approximate} if there exists an isometry \(\tilde f \colon L \to M\) such that \(\tilde f \equiv f \bmod p^a\). 

Starting from a compatible, \(a\)-approximate map \(f\) with \(a \geq 1\), and an arbitrary integer \(b\), the aim of this section is to construct a compatible, \(b\)-approximate map \(f'\) such that \(f' \equiv f \bmod p^a\). We do so by translating the notions of compatibility and \(a\)-approximation into the language of matrices. 

For algorithmic purposes, we choose a basis of \(L\) respecting a chosen Jordan decomposition \(L=\bigoplus_i L_i\) and obtain a block diagonal Gram matrix \(G\) of \(L\). Choosing a basis of \(M\) yields the corresponding gram matrix \(Z\) of \(M\) and a matrix \(F\) representing \(f\) with respect to the two chosen bases. Conventionally, we work with row vectors.

\begin{definition} \label{def:G-compatible}
Let \(G \in \IZ_p^{r \times r}\) be a symmetric, block-diagonal matrix with \(p^i\)-modular blocks \(G_{(i)}\) of size \(r_i \geq 0\).
An invertible matrix \(F \in \IZ_p^{r \times r}\) with blocks~\(F_{(i,j)}\) of size \(r_i \times r_j\) is said to be \emph{\(G\)-compatible} if the following holds for all \(i,j\):
\begin{equation} \label{eq:compatible}
    F_{(i,j)} \equiv 0 \mod{p^{\max(i-j,0)}}.
\end{equation}
\end{definition}

Recall that for a unimodular \(\ZZ_2\)-lattice \(L\), the quantity \(b(v_L,v_L) \bmod 8\) is the oddity of \(L\), where \(v_L\) is its oddity vector (\Cref{def:oddity_vector}).
The oddity of a modular lattice \((L, 2^i q)\) is defined as that of \((L, q)\). The oddity of a symmetric, modular matrix~\(G\) is defined as the oddity of any lattice with Gram matrix~\(G\). The isometry class of a unimodular \(\ZZ_2\)-lattice is determined by its rank, its determinant, its parity and its oddity \cite[Ch. 15, \S 7.3]{conway_sloane:sphere_packings}. 

\begin{definition} \label{def:approximate}
Given an integer \(a \geq 0\), an \emph{\(a\)-approximate} triple \((F,G,Z)\) is a triple of matrices in \(\IZ_p^{r \times r}\) such that \(F\) is \(G\)-compatible, \(Z\) is symmetric, and the following holds for all \(i,j\):
\begin{equation} \label{eq:approximate_all_p}
    (FGF^\intercal - Z)_{(i,j)} \equiv 0 \mod p^{a + \max(i,j)}.
\end{equation}
If \(p = 2\), we require in addition that, for all \(i\) and all \(k \in \{1,\ldots, r_i\}\),
\begin{equation} \label{eq:approximate_p=2_all_a}
    ((FGF^\intercal - Z)_{(i,i)})_{k,k} \equiv 0 \mod 2^{a + i + 1}.
\end{equation}
If \(p = 2\) and \(a = 1\), we also require that, for all \(i\),
\begin{equation} \label{eq:approximate_p=2_a=1}
    v_i (FGF^\intercal - Z)_{(i,i)} v_i^\intercal \equiv 0 \mod 2^{i+3},
\end{equation}
where \(v_i\) is the oddity vector of the matrix \(Z_{(i,i)}\). 
For the sake of brevity, we say that a matrix \(F \in \IZ_p^{r \times r}\) is \(a\)-approximate if the triple \((F,G,G)\) is \(a\)-approximate. 
\end{definition}

\begin{remark} \label{rmk:oddity_G}
For a modular square matrix \(G \in \ZZ_2^{r \times r}\), the oddity vector of \(G\) is given by \(yG^{-1}\) where \(y_k = G_{k,k}\) for \(k \in \{1,\dots,r\}\).
\end{remark}

\begin{remark}\label{rmk:oddity_H}
Let \(p=2\) and set \(H_i=(FGF^\intercal)_{(i,i)}\).
If \((F,G,Z)\) satisfies \cref{eq:approximate_all_p} for \(a \geq 1\), then \(H_i\equiv Z_{(i,i)} \bmod 2^{i+1}\). Therefore, \(H_i\) and \(Z_i\) have the same oddity vector \(v_i\). 
Thus, we can rephrase \cref{eq:approximate_p=2_a=1} by saying that \(H_i\) and \(Z_{(i,i)}\) must have the same oddity as well.
\end{remark}

In what follows, we present three algorithms that enable us to pass from an \(a\)-approximate triple \((F,G,Z)\) with \(a \geq 1\) to a \(b\)-approximate triple \((F',G,Z)\) for an arbitrary \(b\), with appropriate conditions on \(F' - F\). 
We deal with the unimodular case first.

\begin{algorithm}
    \caption{Hensel\_qf\_unimodular\_odd} \label{alg:hensel_modular_odd}
    \begin{algorithmic}[1]
        \REQUIRE \(F,G,Z \in \IZ_p^{r \times r}\), \(a,b \in \IZ\) such that \(p \neq 2\), \(a \geq 1\), \(G\) is unimodular and \((F,G,Z)\) is \(a\)-approximate.
    \ENSURE \(F'\) with \(F' \equiv F \bmod p^a\) such that \((F',G,Z)\) is \(b\)-approximate.
    \WHILE{\(a < b\)}
            \STATE \(A \defeq FGF^\intercal - Z\)
            \STATE \(X \defeq 2^{-1} A \left(GF^\intercal \right)^{-1}\)
        \STATE \(F \defeq F - X\)
        \STATE \(a \defeq 2a\) \label{eq:a=2a}
    \ENDWHILE
    \RETURN \(F\)
    \end{algorithmic}
\end{algorithm}
\begin{proof}[Proof of \Cref{alg:hensel_modular_odd}]
If \(b \leq a\), there is nothing to do. 

Choose \(A = FGF^\intercal - Z\) and \(X = 2^{-1} A \left(GF^\intercal \right)^{-1}\) as in the algorithm. 
Since \((F,G,Z)\) is \(a\)-approximate and \(G\) is unimodular, we have \(A = A^\intercal\), \(A \equiv 0 \bmod{p^{a}}\), and \(X \equiv 0 \bmod{p^a}\). 

We now put \(F' = F - X\). By hypothesis, \(G\) consists of a single block, so \(F'\) is \(G\)-compatible. Moreover, we have
\begin{align*}
    F'GF'^\intercal - Z 
        & = FGF^\intercal - Z - XGF^\intercal - (XGF^\intercal)^\intercal + XGX^\intercal \\
        & = A - 2^{-1} A - 2^{-1} A^\intercal + XGX^\intercal \\
        & = XGX^\intercal \equiv 0 \mod{p^{2a}},
\end{align*}
i.e. \((F',G,Z)\) is a \(2a\)-approximate triple, which justifies Step~\ref{eq:a=2a}.
\end{proof}

\begin{remark} \label{rmk:system_odd}
In \Cref{alg:hensel_modular_odd}, the matrix \(X\) is a particular solution of the linear system \(XGF^\intercal + (XGF^\intercal)^\intercal = A\). Clearly, the kernel of this system consists of matrices of the form \(X = H(GF^\intercal)^{-1}\) with \(H\) anti-symmetric.
\end{remark}

We now turn to \(p = 2\).
Here, the solution of the underlying linear system is less trivial. The following lemma addresses this issue.

\begin{lemma} \label{lem:has_solution}
For any vectors \(b, z \in \FF_2^r\) and any symmetric matrix \(M\in \FF_2^{r \times r}\), there is a solution \(X \in \FF_2^{r \times r}\) to the following linear system 
\begin{align} \label{eq:linear_system_1}
    X + X^\intercal & = M \\
    \label{eq:linear_system_2}
    \textstyle X_{k,k} + \sum_{l = 1}^r X_{k,l} z_l &= b_k, \quad k \in \{1, \dots , r\}
\end{align}
if and only if the diagonal of \(M\) is zero and further
\begin{equation} \label{eq:linear_system_necessary}
    \sum_{k=1}^r z_k^2 b_{k} + \sum_{l < k} z_{k}z_{l}M_{k,l} = 0.
\end{equation}
In this case, the solution space has dimension \(r(r-1)/2+\delta_{z\neq 0}\).
\end{lemma}


\begin{proof}
We follow the Gaussian algorithm. 
Let \(E = \{k \mid z_k = 1\}\) and \(N = \{k \mid z_k = 0\}\).
Up to relabelling, we can suppose that \(E = \{1,\ldots,e\}\) for some \(e \geq 0\).

\Cref{eq:linear_system_1} means \(X_{k,l} + X_{l,k} = M_{k,l}\) for all \(k,l\). 
For \(k = l\), we get the necessary condition \(M_{k,k} = 2X_{k,k} = 0\). 
Moreover, we can choose as pivots all \(X_{k,l}\) with \(k > l\), that is, all variables below the diagonal.

For \(k \in N\), \cref{eq:linear_system_2} reads 
\[
    X_{k,k} + \sum_{l \leq e} X_{k,l} = b_k.
\]
Since it is the only equation in which \(X_{k,k}\) appears, we can solve for \(X_{k,k}\), i.e., we take each \(X_{k,k}\), \(k\in N\), as a pivot.

For \(k \in E\), \cref{eq:linear_system_2} reads
\[
    \sum_{l < k} X_{k,l} + \sum_{k < l \leq e} X_{k,l} = b_k.
\]
By summing up the first \(m\) equations for \(m \leq e\), we obtain the following equivalent set of equations:
\[
    \sum_{l < k \leq m} X_{k,l} + \sum_{k < l \leq m} X_{k,l} + \sum_{k \leq m < l \leq e} X_{k,l} = \sum_{k=1}^{m} b_k, \quad m \in \{1,\ldots,e\}.
\]
Note that the sums correspond to the following entries of the matrix \(X\), here exemplified with \(r = 10\), \(e = 7\), \(m = 5\):
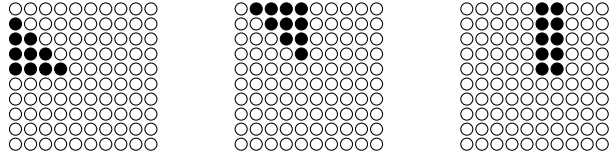
\[
\begin{tikzpicture}[yscale=-0.2,xscale=0.2]
    \begin{scope}[xshift=-15cm]
    \foreach \x in {0,1,...,9}
    \foreach \y in {0,1,...,9}
      \draw (\x,\y) circle (0.4);
    \foreach \P in {(0,1),(0,2),(1,2),(0,3),(1,3),(2,3),(0,4),(1,4),(2,4),(3,4)}
        \fill \P circle (0.4);
    \end{scope}
    \begin{scope}
    \foreach \x in {0,1,...,9}
    \foreach \y in {0,1,...,9}
        \draw (\x,\y) circle (0.4);
    \foreach \P in {(1,0),(2,0),(2,1),(3,0),(3,1),(3,2),(4,0),(4,1),(4,2),(4,3)}
        \fill \P circle (0.4);
    \end{scope}
    \begin{scope}[xshift=+15cm]
    \foreach \x in {0,1,...,9}
    \foreach \y in {0,1,...,9}
      {
        \draw (\x,\y) circle (0.4);
      }
    \foreach \P in {(5,0),(6,0),(5,1),(6,1),(5,2),(6,2),(5,3),(6,3),(5,4),(6,4)}
        \fill \P circle (0.4);
    \end{scope}
\end{tikzpicture}
\]

Substituting into the first sum, and recalling that we are computing in \(\FF_2\), we obtain for each \(m \in \{1,\ldots,e\}\)
\[
    \sum_{l < k \leq m} X_{l,k} + \sum_{k < l \leq m} X_{k,l} + \sum_{k \leq m < l \leq e} X_{k,l} = \sum_{k = 1}^m b_k + \sum_{l < k \leq m} M_{k,l}.
\]
The first two terms are equal (exchange \(k\) and \(l\)). Therefore, we arrive at 
\begin{equation} \label{eq:equivalent_system}
    \sum_{k \leq m < l \leq e} X_{k,l} = \sum_{k = 1}^m b_k + \sum_{l < k \leq m} M_{k,l}, \quad m \in \{1,\ldots,e\}.
\end{equation}
For \(m = e\), we have in particular
\[
    \sum_{k=1}^e b_{k} + \sum_{l<k \leq e}M_{k,l} = 0,
\]
which show the necessity of \cref{eq:linear_system_necessary}. 
Furthermore, \(X_{k,k+1}\) appears only in \cref{eq:equivalent_system} with \(m = k\). Thus, we can take the additional \(e-1\) pivots \(X_{k,k+1}\), \(k \in E\). We have used up all of the \(r(r+1)/2\) linear equations and obtained 
a total of \(e-1+r(r-1)/2+r-e=r(r+1)/2 -\delta_{e>0}\) pivots. 
This shows that the solution space is of dimension \(r(r-1)/2+\delta_{e>0}\).
For \(r = 10\) and \(e = 7\), the pivots are displayed in the following matrix.

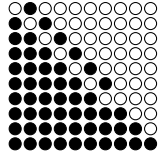
\[
\begin{tikzpicture}[yscale=-0.2,xscale=0.2]
    \foreach \x in {0,1,...,9}
    \foreach \y in {0,1,...,9}
      {
        \draw (\x,\y) circle (0.4);
      }
    \foreach \k in {0,1,...,5}
      {
        \fill ({\k+1},\k) circle (0.4);
      }
    \foreach \k in {7,8,9}
      {
        \fill (\k,\k) circle (0.4);
      }
    \fill (0,1) circle (0.4);
    \foreach \k in {1,...,8}
    \foreach \l in {0,1,...,\k}
         \fill (\l,{\k+1}) circle (0.4);
\end{tikzpicture} \qedhere
\]
\end{proof}

\begin{algorithm}
    \caption{Hensel\_qf\_unimodular\_even}\label{alg:hensel_modular_even}
    
    \begin{algorithmic}[1]
        \REQUIRE \(F, G, Z \in \IZ_2^{r\times r}\), \(a,b \in \IZ\) such that \(a \geq 1\), \(G\) is unimodular and \((F,G,Z)\) is \(a\)-approximate. 
        \ENSURE \(F'\) with \(F' \equiv F \bmod 2^a\) such that \((F',G,Z)\) is \(b\)-approximate. 
        \IF{\(a = 1\)}
            \STATE \(A \defeq FGF^\intercal - Z\)
            \STATE solve the following linear system for \(X \in \IZ_2^{r \times r}\) \label{eq:LS_Hensel_even}
            \begin{align*}
                    X + X^\intercal & \equiv 2^{-1}A   \mod 2 \\
                \textstyle X_{k,k} + \sum_{l=1}^r X_{k,l} (Z^{-1})_{l,l} & \equiv 2^{-2}A_{k,k}  \mod 2, \quad k \in \{1,\ldots,r\}
            \end{align*}
            \STATE \(Y \defeq 2X(GF^\intercal)^{-1}\)
            \STATE \(F \defeq F - Y\)
            \STATE \(a \defeq 2\) \label{eq:a=2}
        \ENDIF
        \WHILE{\(a < b\)}
            \STATE \(A \defeq FGF^\intercal - Z\) \label{step:A=FGF-Z}
            \STATE find \(L\) strictly upper triangular and \(D\) diagonal
                   with \(A = L + L^\intercal  + D\)
            \STATE \(X \defeq (L + 2^{-1}D) \left(GF^\intercal \right)^{-1}\)
            \STATE \(F \defeq F - X\)
            \STATE \(a \defeq 2a-1\) \label{eq:a=2a-1}
        \ENDWHILE
        \RETURN \(F\)
    \end{algorithmic}
\end{algorithm}
\begin{proof}[Proof of \Cref{alg:hensel_modular_even}]
If \(b \leq a\), there is nothing to do. 

Suppose first that \((F,G,Z)\) is an \(a\)-approximate triple with \(a \geq 2\). Put \(A = FGF^\intercal - Z\) as in Step \ref{step:A=FGF-Z} of the algorithm. Clearly, there exist a unique strictly upper triangular matrix \(L\) and a unique diagonal matrix \(D\) with \(A = L + L^\intercal + D\). By \cref{eq:approximate_all_p}, it holds \(L \equiv 0 \bmod 2^{a}\) and, by \cref{eq:approximate_p=2_all_a}, \(D \equiv 0 \bmod 2^{a+1}\). 

Set \(X = (L+2^{-1}D)(GF^\intercal)^{-1}\) and \(F' = F - X\), as in the algorithm.
By hypothesis, \(G\) consists of a single unimodular block, so \(F'\) is \(G\)-compatible, and \(F' \equiv F \bmod 2^a\). Moreover, we have
\begin{align*}
    F'GF'^\intercal - Z 
        & = FGF^\intercal - Z - XGF^\intercal - (XGF^\intercal)^\intercal + XGX^\intercal \\
        & = A - (L+2^{-1}D) - (L^\intercal +2^{-1}D) + XGX^\intercal \\
        & = XGX^\intercal \equiv 0 \mod 2^{2a},
\end{align*}
i.e., \((F',G,Z)\) is a \((2a-1)\)-approximate triple, which justifies Step~\ref{eq:a=2a-1}.

Suppose now that \((F,G,Z)\) is a \(1\)-approximate triple. As in the algorithm, we set \(A = FGF^\intercal - Z\).

The linear system in Step~\ref{eq:LS_Hensel_even} is in fact a linear system over \(\FF_2\) as in \Cref{lem:has_solution} with \(M = 2^{-1}A\), \(b_{k} = 2^{-2}A_{k,k}\) and \(z_k = (Z^{-1})_{k,k}\) for \(k \in \{1,\dots, r\}\).
By~\Cref{lem:has_solution}, a sufficient and necessary condition for a solution to exist is \cref{eq:linear_system_necessary}, which here can be equivalently written as
\begin{equation} \label{eqn:hassolution}
	zAz^\intercal \equiv 0 \mod 8.
\end{equation}
By \Cref{rmk:oddity_G}, \(z Z\) is an oddity vector for \(Z^{-1}\), hence \(zZz^{\intercal} = zZ Z^{-1} (Z z)^\intercal\) is the oddity of \(Z^{-1}\). 
On the other hand, \(Z\) and \(Z^{-1}\) have the same oddity, because the are both Gram matrices for the same lattice with respect to different bases.
Since \(FGF^\intercal \equiv Z \bmod 2\) by \cref{eq:approximate_all_p}, the same argument yields that \(z FGF^\intercal z^\intercal\) is the oddity of \(FGF^\intercal\). As \((F,G,Z)\) is \(1\)-approximate, \(FGF^\intercal\) and \(Z\) have the same oddity by \Cref{rmk:oddity_H}. Therefore, \cref{eqn:hassolution} holds true and the linear system in Step~\ref{eq:LS_Hensel_even} can be solved.
 
Set \(Y = 2X(GF^\intercal)^{-1}\) and \(F' = F - Y\), as in the algorithm. By hypothesis, \(G\) consists of a single unimodular block, so \((F',G)\) is a compatible pair and \(F' \equiv F \bmod 2\). We claim that \((F',G,Z)\) is \(2\)-approximate, in order to justify Step~\ref{eq:a=2}. Indeed, we have
\begin{align*}
    F'GF'^\intercal - Z
        & = FGF^\intercal - Z - YGF^\intercal - (YGF^\intercal)^\intercal + YGY^\intercal \\
        & = A - 2X - 2X^\intercal + YGY^\intercal \equiv 0 \mod 4.
 \end{align*}
Finally, we need to show that the diagonal of \(F'GF'^\intercal - Z\) vanishes modulo \(8\). Indeed, we have
 \begin{align*}
    (F'GF'^\intercal - Z)_{k,k} & = (A - 2X - 2X^\intercal + YGY^\intercal)_{k,k} \\
         & = A_{k,k} - 2X_{k,k} - 2(X^\intercal)_{k,k} + 4(X(FGF^\intercal)^{-1}X^\intercal)_{k,k} \\
         & \equiv A_{k,k} - 4X_{k,k} + 4(XZ^{-1}X^\intercal)_{k,k} \mod 8 \\
         & = A_{k,k} - 4X_{k,k} + 4\sum_{l,m=1}^r X_{k,l}(Z^{-1})_{l,m}(X^\intercal)_{m,k} \\
         & \equiv A_{k,k} - 4X_{k,k} + 4\sum_{l=1}^r X_{k,l}^2 (Z^{-1})_{l,l} \mod 8 \\
         & \equiv A_{k,k} - 4X_{k,k} + 2^{i+2}\sum_{l=1}^r X_{k,l} (Z^{-1})_{l,l} \mod 8
 \end{align*}
for all \(k \in \{1,\ldots,r\}\), because \(X_{k,l}^2 \equiv X_{k,l} \bmod 2\). Therefore, we have 
\[
    (F'GF'^\intercal - Z)_{k,k} \equiv 0 \mod 8
\] by the defining equations of~\(X\).
\end{proof}

\begin{algorithm}
    \caption{Hensel\_qf} \label{alg:hensel}

    \begin{algorithmic}[1]
        \REQUIRE \(F,G,Z \in \IZ_p^{r \times r}\), \(a,b \in \IZ\) such that \(a \geq 1\) and \((F, G, Z)\) is \(a\)-approximate.
        \ENSURE \(F' \in \IZ_p^{r \times r}\) with
        \begin{equation} \label{eq:F'-F}
            (F'-F)_{(i,j)} \equiv 0 \mod p^{a + \max(i-j,0)}
        \end{equation}
        for all \(i,j\), and such that \((F',G,Z)\) is \(b\)-approximate.
    \STATE \(i_1 \defeq \min\{i \mid G_{(i)} \neq 0\}\)
    \STATE \(G_1 \defeq G_{(i_1)}\)
    \STATE Define \(G_2\), and \(F_{ij}\), \(Z_{ij}\) for \(i,j \in \{1,2\}\), so that \(F_{11}\) and \(Z_{11}\) have the same size as \(G_1\) and, moreover,
    \[
        F = \begin{pmatrix} F_{11} & F_{12} \\ F_{21} & F_{22} \end{pmatrix} \quad 
        G = \begin{pmatrix} G_{1} & 0 \\ 0 & G_{2} \end{pmatrix} \quad
        Z = \begin{pmatrix} Z_{11} & Z_{12} \\ Z_{21} & Z_{22} \end{pmatrix}
    \]
    \STATE \(Z_{11}' \defeq Z_{11} - F_{12} G_2 F_{12}^\intercal\)
    \STATE \(F_{11} \defeq \mathrm{Hensel\_qf\_unimodular}(F_{11}, p^{-i_1} G_1, p^{-i_1} Z_{11}', a,b)\) \label{step:F11}
    \IF{\(G_2 = 0\)}
        \RETURN \(F\)
    \ENDIF
    \STATE \(s \defeq \min\{j - i_1 \mid j>i_1,\,G_{(j)} \neq 0\}\)
    \WHILE{\(a < b\)}
        \STATE \(Z_{22}' \defeq Z_{22} - F_{21} G_1 F_{21}^\intercal\)
        \STATE \(F_{22} \defeq \mathrm{Hensel\_qf}(F_{22}, G_2, Z_{22}', a, a+s)\) \label{step:F22}
        \STATE \(F_{21} \defeq (Z_{21} - F_{22} G_2 F_{12}^\intercal)(G_1 F_{11}^\intercal)^{-1}\) \label{step:F21}
    \STATE \(a \defeq a + s\) \label{step:a=a+s}
    \ENDWHILE
    \RETURN \(F\)
    \end{algorithmic}
\end{algorithm}

\begin{proof}[Proof of Algorithm \ref{alg:hensel}]
If \(b \leq a\), there is nothing to do. 
After defining \(i_1\), \(G_1\) and all other submatrices as in the algorithm, we have
\[
    F G F^\intercal - Z = \begin{pmatrix}
        F_{11} G_{1} F_{11}^\intercal + F_{12} G_2 F_{12}^\intercal - Z_{11} &
        F_{11} G_{1} F_{21}^\intercal + F_{12} G_2 F_{22}^\intercal - Z_{12} \\
        F_{21} G_{1} F_{11}^\intercal + F_{22} G_2 F_{12}^\intercal - Z_{21} &
        F_{21} G_{1} F_{21}^\intercal + F_{22} G_2 F_{22}^\intercal - Z_{22}
    \end{pmatrix}
\]
We put \(Z_{11}' = Z_{11} - F_{12} G_2 F_{12}^\intercal\) and \(Z_{22}' = Z_{22} - F_{21} G_1 F_{21}^\intercal\), as in the algorithm. Since \(G_2 \equiv 0 \bmod p^{i_1+1}\) and \(F\) is \(G\)-compatible, \(Z_{11}\equiv Z_{11}' \bmod p^{i_1+1}\) and \((Z_{22})_{(i)} \equiv (Z_{22}')_{(i)} \bmod p^{i+1}\).
Therefore, when \(p =2\), the oddity vectors do not change when passing from \(Z\) to \(Z'\) (although the oddity might change). 

Looking at the diagonal blocks of \(F G F^\intercal - Z\), we observe that from the fact that \((F,G,Z)\) is \(a\)-approximate, it follows that both \((F_{11}, G_1, Z_{11}')\) and \((F_{22}, G_2, Z_{22}')\) are \(a\)-approximate. 
Hence, we are entitled to apply \Cref{alg:hensel_modular_odd} or \Cref{alg:hensel_modular_even} to \((F_{11}, p^{-i_1} G_1, p^{-i_1} Z_{11}')\), given that \(G_1\) is \(p^{i_1}\)-modular by definition. 
In this way, we obtain at Step~\ref{step:F11} a matrix \(F_{11}'\) with 
\begin{equation} \label{eq:F11'-F11}
    F_{11}' - F_{11} \equiv 0 \bmod p^a
\end{equation} 
such that \((F_{11}', G_1, Z_{11}')\) is \(b\)-approximate.

If \(G_2 = 0\), that is, \(G = G_1\) and \(Z = Z_{11} = Z_{11}'\), then there is nothing else to do. 
Otherwise, \(s\) is well defined and \(G_2\) has less modular blocks by than \(G\). 
By induction on the number of blocks, we obtain at Step~\ref{step:F22} a matrix \(F_{22}'\) with 
\begin{equation} \label{eq:F22'-F22}
    (F_{22}' - F_{22})_{(i,j)} \equiv 0 \mod p^{a + \max(i-j,0)}
\end{equation} 
for all \(i,j\), such that \((F_{22}', G_2, Z_{22}')\) is \((a+s)\)-approximate. 
As in the algorithm, we put \(F_{21}' = (Z_{21} - F_{22}' G_2 F_{12}^\intercal)(G_1 F_{11}'^\intercal)^{-1}\) and 
\[
    F' = \begin{pmatrix} F_{11}' & F_{12} \\ F_{21}' & F_{22}' \end{pmatrix}.
\]

We first show that \cref{eq:F'-F} holds. Given \cref{eq:F11'-F11,eq:F22'-F22}, we only need to prove that, for all \(i > i_1\),
\[
    (F_{21}' - F_{21})_{(i,i_1)} \equiv 0 \mod p^{a + i - i_1}.
\]

Indeed, it holds \(G_2 \equiv 0 \bmod p^{i}\), so again by \cref{eq:F11'-F11,eq:F22'-F22}, we get
\begin{align*}
    (F_{21}' - F_{21})_{(i,i_1)}
        & = (Z_{21} - F_{22}' G_2 F_{12}^\intercal  - F_{21} G_1 F_{11}'^\intercal)_{(i,i_1)}(G_1 F_{11}'^\intercal)^{-1} \\
        & \equiv (Z_{21} - F_{22} G_2 F_{12}^\intercal  - F_{21} G_1 F_{11}^\intercal)_{(i,i_1)}(G_1 F_{11}'^\intercal)^{-1} \mod p^{a + i - i_1} \\
        & \equiv (Z-FGF^\intercal)_{(i,i_1)} (G_{1} F_{11}'^\intercal)^{-1} \mod p^{a + i - i_1}
\end{align*}
which is equivalent to \(0 \bmod p^{a + i -i_1}\), since \((F,G,Z)\) is \(a\)-approximate. 

It only remains to prove that \((F', G, Z)\) is \((a+s)\)-approximate in order to justify Step~\ref{step:a=a+s}. 
We define \(Z_{22}'' = Z_{22} - F_{21}' G_1 F_{21}'^\intercal\). 
By definition of \(F_{21}'\), we have
\[
    F' G F'^\intercal - Z = \begin{pmatrix}
    F'_{11} G_1 F_{11}'^\intercal - Z_{11}' & 0 \\ 0 & F_{22}' G_2 F_{22}'^\intercal -Z_{22}''
    \end{pmatrix}
\]
Thus, we need to prove equivalently that \((F_{22}', G_2, Z_2'')\) is \((a+s)\)-approximate. 

We define the matrices \(A = F_{22}'G_2F_{22}'^\intercal - Z_{22}'\), \(B = (F_{21}' - F_{21}) G_1 F_{21}'^\intercal\) and \(C = (F'_{21} - F_{21}) G_1 (F_{21}^\intercal - F_{21}'^\intercal)\), so that
\[
    F_{22}'G_2F_{22}'^\intercal - Z_{22}'' = A +  B + B^\intercal + C.
\]

In order to show \cref{eq:approximate_all_p} for \((F_{22}', G_2, Z_{22}'')\), we estimate the four terms separately. 
Since \((F_{22}', G_2, Z_{22}')\) is \((a+s)\)-approximate by construction, we have \(A_{(i,j)} \equiv 0 \bmod p^{a+s+\max(i,j)}\) for all \(i,j\).
By \cref{eq:F'-F}, we have
\begin{equation} \label{eq:B_oddity}
    B_{(i,j)} = (F_{21}' - F_{21})_{(i,i_1)} G_1 (F_{21}'^\intercal)_{(i_1,j)} \equiv 0 \mod p^\beta
\end{equation}
with
\[
    \beta = a + \max(i-i_1,0) + i_1 + \max(j-i_1,0) \geq a + s + \max(i,j).
\]
Further, we have
\begin{equation} \label{eq:C_oddity}
    C_{(i,j)} = (F_{21}' - F_{21})_{(i,i_1)} G_1 (F_{21}^\intercal - F_{21}'^\intercal)_{(i_1,j)} \equiv 0 \mod p^\gamma
\end{equation}
with
\begin{align*}
    \gamma & = a + \max(i - i_1,0) + i_1 + a + \max(j - i_1,0)  \\
      & = 2a + i + j - i_1 \geq a + s + \max(i,j) + 1.
\end{align*}
(Note that, since \(a \geq 1\), it holds \(2a \geq a +1\).)

Suppose now that \(p = 2\). In order to show \cref{eq:approximate_p=2_all_a} for \((F_{22}',G_2,Z_{22}'')\), we estimate the three terms \(A\), \(B + B^\intercal\) and \(C\) separately. Since \((F_{22}', G_2, Z_{22}')\) is \((a+s)\)-approximate by construction, we have \((A_{(i,i)})_{k,k} \equiv 0 \bmod 2^{a+s+i+1}\) for all \(i\) and all \(k \in \{1,\ldots,r_i\}\). 
By \cref{eq:B_oddity}, it holds
\[
((B+B^\intercal)_{(i,i)})_{k,k} \equiv 2(B_{(i,i)})_{k,k} \mod 2^{a+s+i+1}.
\]
Finally, \cref{eq:C_oddity} implies that also \((C_{(i,i)})_{k,k} \equiv 0 \bmod 2^{a+s+i+1}\). 

As \(a+s \geq 2\), we do not need to check \cref{eq:approximate_p=2_a=1}.
This shows that \((F_{22}', G_2, Z_{22}'')\) is \((a+s)\)-approximate, and concludes the proof.
\end{proof}

\section{Generators of orthogonal groups} \label{sec:generators}

Let \((L,q)\) be an integral \(\ZZ\)-lattice. 
If \(L\) is even, then its discriminant group \(L^\sharp = L^\vee/L\) carries the \(\QQ/\ZZ\) valued \emph{discriminant quadratic form}
\[
	q^\sharp(x+ L) = q(x) + \ZZ.
\]
If \(L\) is odd, then \(L^\sharp\) carries the \(\QQ/\ZZ\)-valued \emph{discriminant bilinear form}
\[
	b^\sharp(x+L,y+L) = b(x,y) + \ZZ.
\]
Note that the Sylow decomposition \(L^\sharp  = \bigoplus_p L^\sharp_p\) is orthogonal with respect to \(b^\sharp\), where \(L^\sharp_p\) is the discriminant group of \(L\otimes \ZZ_p\). 
This results in a decomposition of \(O(L^\sharp)\) as a direct product, and reduces the problem of finding generators for \(O(L^\sharp)\) for a \(\IZ_p\)-lattice.

For an integral \(\IZ_p\)-lattice \(L\) with a Jordan decomposition \(L = \bigoplus_{i = 0}^n L_i\), the map \(O(L/p^nL) \to O(L^\sharp)\) is surjective by a result of Nikulin \cite[Cor.~1.9.6, and Thm.~1.16.4]{nikulin:bilinear}. 
In this section, we study the structure of \(O(L)\) and of the finite groups \(O(L/p^nL)\). Moreover, we provide a small set of generators for the latter.
(It is known that \(O(L)\) is generated by reflections and Eichler transformations. However, the number of all such transformations is too big to be useful for computations on the finite group \(O(L/p^nL)\).)

The following technical lemma is needed for the proof of \Cref{prop:surj}.
\begin{lemma}\label{lem:strange}
    Let \(G \in \IZ_2^{r \times r}\) be a unimodular symmetric matrix, and let \(F \in \IZ_2^{r \times r}\) be an invertible matrix such that \(F G F^\intercal \equiv G \bmod 2\). 
If \(h \in \IZ_2^r\) is the vector given by \(h_k = 2^{-1}(F G F^\intercal - G)_{k,k}\) for \(k \in \{1,\ldots,r\}\), then
\[
    h (FGF^\intercal)^{-1} h^\intercal \equiv 0 \mod 4.
\]
\end{lemma}
\begin{proof}
Consider the \(2\)-elementary quadratic form \(q \colon (\IZ/2\IZ)^r \rightarrow \IQ/2\IZ\) defined through the isomorphism \(\IZ_2/2\IZ_2 \cong \IZ/2\IZ\) by the formula
\(q(x) = \frac12 xGx^\intercal \bmod 2\),
and let \(b\) be its associated bilinear form given by \(b(x,y) = xGy^\intercal \bmod 2\).

By hypothesis, the matrix \(F\) defines an isometry \(f \in O(b)\) with respect to the standard basis.
The vector \(hG^{-1}\) represents the defect \(v_f\) of \(f\) (\Cref{def:defect}).

Consider the vector \(w = h(GF^\intercal)^{-1}\). 
A straightforward computation shows that \(h(FGF^\intercal)^{-1}h^\intercal = wGw^\intercal\).
Since \(GF^\intercal \equiv F^{-1}G \bmod 2\), we have \(w \equiv hG^{-1}F \bmod 2\), so \(w\) represents \(fv_f\).
By \Cref{cor:bijection}, it holds \(q(fv_f)=0\), that is, \(wGw^{\intercal} \equiv 0 \bmod 4\), which is what we wanted to prove.
\end{proof}

Let \((L,q) = \bigoplus_i (L_i, p^i q_i)\) be a Jordan decomposition of a \(\IZ_p\)-lattice.
Following Conway and Sloane \cite{conway-sloane:mass_formula}, we say that \(L_i\) is \emph{free} if both lattices \((L_{i-1}, q_{i-1})\) and \((L_{i+1}, q_{i+1})\) are even.
Otherwise, we say that \(L_i\) is \emph{bound}. Note that the zero lattice is even. For \(p \neq 2\), all constituents are free.

We need a natural definition of the \(p\)-elementary forms \(\bar{q}_i\) and \(\bar{b}_i\) which is independent of the chosen Jordan decomposition.
For \(\bar{b}_i\) this works out, but for \(\bar{q}_i\) to be naturally defined, we need \(L_i\) to be free. 

\begin{definition} \label{def:rho_i(L)}
Set \(H_i = (p^i L^\vee \cap L)+p L\) and define 
\[
	\rho_i(L) = H_i/H_{i+1}.
\]
If \(L_i\) is free, we equip \(\rho_i(L)\) with the \(p\)-elementary quadratic form
\[
	\bar{q}_i(x + H_{i+1}) = q_i(x) \mod p,
\] 
whose bilinear form is non-degenerate and given by 
\[
	\bar{b}_i(x+H_{i+1},y+H_{i+1}) = b_i(x,y) \mod p.
\] 
If \(L_i\) is bound, then \(\bar{q}_i\) is not well defined and we equip \(\rho_i(L)\) merely with the bilinear form \(\bar{b}_i\).
\end{definition}

\begin{remark}
We have that \(H_k = \bigoplus_{i<k}p L_i \oplus \bigoplus_{i\geq k}L_i\) for any Jordan decomposition of \(L\). 
Therefore, \(\rho_i(L) \cong L_i/pL_i\). This shows that the two definitions of \(\bar{q}_i\) in \Cref{def:qi,def:rho_i(L)} are canonically isomorphic if \(L_i\) is free.
Note that if \(L_i\) is bound, then \(\bar{q}_i\) as in \Cref{def:qi} depends on the choice of a Jordan decomposition.
\end{remark}  

\begin{proposition} \label{prop:surj}
For any \(\ZZ_p\)-lattice \((L,q)\), the natural homomorphism
\[
    O(L,q) \rightarrow \prod_i O(\rho_i(L))
\]
is surjective.
\end{proposition}
\begin{proof}
Choose a Jordan decomposition \((L,q) = \bigoplus_i (L_i, p^i q_i)\).
After rescaling the form \(q\), we may assume that \(L\) is integral, i.e., \((L,q) = \bigoplus_{i \geq 0} (L_i, p^i q_i)\). 
Choose a basis for \(L\) respecting the given Jordan decomposition, and let \(G \in \IZ_p^{r \times r}\) be the Gram matrix of \(L\) with respect to this basis, with modular blocks~\(G_{(i)} \in \IZ_p^{r_i \times r_i}\), where \(r\) and \(r_i\) denote the ranks of \(L\) and \(L_i\), respectively.

With respect to the chosen basis, any element \(f=\bigoplus_i f_i \in \prod_i O(\rho_i(L))\) is represented by a \(G\)-compatible, block-diagonal matrix \(F \in \IZ_p^{r \times r}\) with blocks \(F_{(i,i)} \in \IZ_p^{r_i \times r_i}\). We put \(A = FGF^\intercal - G\), which is also a block-diagonal matrix.
Recall that \(\rho_i(L)\) is endowed with a symmetric bilinear form~\(\bar{b}_i\) (\Cref{def:rho_i(L)}). By definition of \(O(\rho_i(L))\), the following holds for all \(i\):
\begin{equation} \label{eq:approximate_all_p_(FGG)}
    A_{(i,i)} \equiv 0 \mod p^{i + 1}.
\end{equation}

In particular, if \(p \neq 2\), the matrix \(F\) is \(1\)-approximate (\Cref{def:approximate}).
Hence, we can apply \Cref{alg:hensel} in order to find a convergent sequence of matrices \((F_n)_{n\in \NN}\), whose limit in \(\IZ_p^{r \times r}\) represents the desired preimage of \(f\). 

If \(p = 2\), \(F\) may not be \(1\)-approximate due to \cref{eq:approximate_p=2_all_a} or \cref{eq:approximate_p=2_a=1}.
In order to be able to apply \Cref{alg:hensel} and conclude as before, we need to find another \(G\)-compatible (not necessarily block-diagonal) matrix \(F'\) such that \(F'_{(i,i)} = F_{(i,i)}\) for all \(i\), and such that \(F'\) is \(1\)-approximate.

Since it suffices to lift generators, we may assume that \(f\) is of the form \(f_n \oplus \bigoplus_{i \neq n}\id_{L_i}\) with \(f_n \in O(\rho_n(L))\) for some \(n \geq 0\), i.e., we can assume that \(F_{(i,i)} = I\) for all \(i \neq n\), where \(I\) denotes the identity matrix.
Then, the entries of \(A\) off the block diagonal \(A_{(n,n)} = F_{(n,n)} G_{(n)} F_{(n,n)}^\intercal - G_{(n)}\) are equal to zero.

Suppose first that the Jordan component \(L_n\) is free, that is, both \(L_{n-1}\) and \(L_{n+1}\) are even.
In this case, \(\rho_n(L)\) is endowed with a quadratic form~\(\bar{q}_n\). 
By definition of \(O(\rho_i(L))\), following holds for all \(i\) and all \(k \in \{1,\ldots,r_i\}\):
\begin{equation*} \label{eq:approximate_p=2_all_a_(F,G,G)} 
    (A_{(i,i)})_{k,k} \equiv 0 \mod 2^{i+2}.
\end{equation*}
Since \(\bar f_i\) preserves the form~\(\bar{b}_i\), we have \(v_i \equiv v_{i}F_{(i,i)} \bmod 2\). Thus, the latter represents an oddity vector for \(L_i\) as well, which implies  \cref{eq:approximate_p=2_a=1}.
Hence, \(F\) is \(1\)-approximate, and we conclude.

Now, suppose that \(L_n\) is bound.
For later use, we define \(h \in \IZ_2^{r_n}\) by:
\begin{equation}\label{eq:def_h}
    h_k = 2^{-n-1}(A_{(n,n)})_{k,k}
\end{equation}
which is indeed integral by \cref{eq:approximate_all_p_(FGG)}.
Since \(hG_{(n)}^{-1}\) represents the defect \(w=v_{\bar f_n}\) of~\(\bar f_n\), we have
\begin{equation}\label{eq:vh}
v_n h^\intercal = v_n G G^{-1}h^\intercal=b_n(v_n,w)\equiv b_n(w,w) \equiv 0 \mod 2,
\end{equation}
where the last congruence follows from \Cref{cor:bijection}.

Let us start with the case that \(L_{n-1}\) is odd. Without loss of generality, we assume that \(G\) consists of only two modular blocks, \(G_{(n-1)}\) and \(G_{(n)}\). 
Then, there exists a vector \(v \in \IZ_2^{r_{n-1}}\) with
\begin{equation} \label{eq:L_n-1_odd}
    vG_{(n-1)}v^\intercal \equiv 2^{n-1} \mod 2^n.
\end{equation}

Furthermore, we define 
\begin{align*}
    F'_{(n,n-1)} &= 2h^\intercal v, \\
    F_{(n-1,n)}' &= -G_{(n-1)}F_{(n,n-1)}'^\intercal(G_{(n)} F_{(n,n)}^\intercal)^{-1},
\end{align*}
and
\[
    F' = \begin{pmatrix} I & F'_{(n-1,n)} \\ F'_{(n,n-1)} & F_{(n,n)} \end{pmatrix}.
\]
Note that \(F'\) is a \(G\)-compatible matrix, because \(F_{(n,n-1)}' \equiv 0 \bmod 2\).

We claim that \(F'\) is \(1\)-approximate. We define the matrix \(A' = F'GF'^\intercal - G\) and we compute
\[
    A' = 
    \begin{pmatrix}
    F_{(n-1,n)}' G_{(n)} F_{(n-1,n)}'^\intercal & 0 \\ 0 & F_{(n,n-1)}' G_{(n-1)} F_{(n,n-1)}'^\intercal - A_{(n,n)}
    \end{pmatrix}
\]
Using \cref{eq:approximate_all_p_(FGG)}, it is straightforward to check that
\begin{equation*}
    A'_{(i,i)} \equiv 0 \bmod 2^{i+1}
\end{equation*}
for both \(i = n-1\) and \(i = n\). 
Hence, \cref{eq:approximate_all_p} holds for \((F',G,G)\) with \(a = 1\).

We compute
\begin{align*}
    A'_{(n-1,n-1)} & = G_{(n-1)} 2v^\intercal h (F_{(n,n)}G_{(n)}F_{(n,n)}^\intercal)^{-1}2h^\intercal v G_{(n-1)} \\
        & = 4G_{(n-1)} v^\intercal h (F_{(n,n)}G_{(n)}F_{(n,n)}^\intercal)^{-1}h^\intercal v G_{(n-1)} \\
        & = 4h (F_{(n,n)}G_{(n)}F_{(n,n)}^\intercal)^{-1}h^\intercal G_{(n-1)} v^\intercal v G_{(n-1)}.
\end{align*}
By~\Cref{lem:strange}, it holds \(2^{n}h (F_{(n,n)}G_{(n)}F_{(n,n)}^\intercal)^{-1}h^\intercal \equiv 0 \bmod 4\), so 
\begin{equation} \label{eqn:A'}
    A'_{(n-1,n-1)} \equiv 0 \mod 2^{n+2}.
\end{equation}
Further, we compute
\begin{equation} \label{eq:A'_n-1odd}
    A'_{(n,n)} = 4 v G_{(n-1)}v^\intercal \cdot h^\intercal h - A_{(n,n)}.
\end{equation}
Using \cref{eq:L_n-1_odd} we obtain that
\(A'_{(n,n)}\equiv 2^{n+1} \cdot h^\intercal h - A_{(n,n)} \bmod 2^{n+2}\).
By the definition of \(h_k\) in \cref{eq:def_h}, this yields for the diagonal
\[
    (A'_{(n,n)})_{k,k}  \equiv 2^{n+1}h_k^2 - 2^{n+1} h_k \equiv 2^{n+1}h_k(h_k-1) \equiv 0 \mod 2^{n+2}.
\]
Hence, \cref{eq:approximate_p=2_all_a} holds for \((F',G,G)\) with \(a = 1\).
By \cref{eqn:A'}, 
\[
    v_{n-1} A'_{(n-1,n-1)}v_{n-1}\equiv 0 \mod 2^{n+2}.
\] 
\Cref{eq:vh,eq:A'_n-1odd} combined yield that
\[
    v_n A'_{(n,n)}v_n^\intercal=4 v G_{(n-1)}v^\intercal \cdot (h v_n^\intercal)^2 - v_n A_{(n,n)} v_n^\intercal \equiv -v_n A_{(n,n)} v_n^\intercal \mod 2^{n+3}.
\]
Since \(F_{(n,n)} \bmod 2\) represents the isometry \(f_n\) of the bilinear form~\(\bar{b}_n\), it fixes the oddity vector \(v_n\). Therefore, \(-v_n A_{(n,n)} v_n^\intercal \equiv 0 \bmod 2^{n+3}\).
Hence, \cref{eq:approximate_p=2_a=1} holds for \((F',G,G)\). This shows that \(F'\) is \(1\)-approximate.

Now, we turn to the case that \(L_{n+1}\) is odd. Without loss of generality, we assume that \(G\) consists of only two modular blocks, \(G_{(n)}\) and \(G_{(n+1)}\). Then, there exists a vector \(v \in \IZ_2^{r_{n+1}}\) with
\begin{equation} \label{eq:L_n+1_odd}
    vG_{(n+1)}v^\intercal \equiv 2^{n+1} \mod 2^{n+2}.
\end{equation}
By \cref{eq:approximate_all_p_(FGG)}, we can define \(h \in \IZ_2^{r_n}\) in the following way:
\[
    h_k = 2^{-n-1}(A_{(n,n)})_{k,k}.
\] 
Furthermore, we define 
\begin{align*}
    F'_{(n,n+1)} &= h^\intercal v, \\
    F_{(n+1,n)}' &= -G_{(n+1)}F_{(n,n+1)}'^\intercal(G_{(n)} F_{(n,n)}^\intercal)^{-1},
\end{align*}
and
\[
    F' = \begin{pmatrix} F_{(n,n)} & F'_{(n,n+1)} \\ F'_{(n+1,n)} & I \end{pmatrix}.
\]
Note that \(F'\) is a \(G\)-compatible matrix, because \(F_{(n+1,n)}' \equiv 0 \bmod 2\).

We claim that \(F'\) is \(1\)-approximate. 
The proof is analogous to the case of \(L_{n-1}\) odd. 
We define the matrix \(A' = G-F'GF'^\intercal\) and we compute
\[
    A' = 
    \begin{pmatrix}
    F_{(n,n+1)}' G_{(n+1)} F_{(n,n+1)}'^\intercal - A_{(n,n)} & 0 \\ 0 & F_{(n+1,n)}' G_{(n)} F_{(n+1,n)}'^\intercal 
    \end{pmatrix}
\]
Similar to before, we have
\[
    A'_{(n+1,n+1)}= h(F_{(n,n)} G_{(n)} F_{n,n)}^\intercal)^{-1} h^\intercal \cdot G_{(n+1)}(v^\intercal v) G_{(n+1)} \equiv 0 \mod 2^\alpha
\]
with \(\alpha=(2-n)+(n+1) + 0 + (n+1)= n+3\) by \Cref{lem:strange}. 
One calculates that 
\[A'_{(n,n)} = v G_{(n+1)}v^\intercal \cdot (h^\intercal h)-A_{(n,n)},\]
and, using \cref{eq:L_n+1_odd} as well as the definition of \(h_k\),
\[(A'_{(n,n)})_{k,k} \equiv 2^{n+1} h_k(h_k-1) \equiv 0 \mod 2^{n+2}.\]
This shows that \cref{eq:approximate_p=2_all_a} holds.
For \cref{eq:approximate_p=2_a=1}, we use \cref{eq:vh} and the fact that \(f(v_n)\equiv v_n \bmod 2\).
\end{proof}

Let \((L,q)\) be an integral \(\ZZ_p\)-lattice of rank \(r\). 
We now explain how to obtain generators for the group \(O(L/p^nL)\) from the proofs of \Cref{prop:order_orthogonal_elementary,prop:surj} and the quadratic Hensel lifting algorithm (\Cref{alg:hensel}). 

Choose a basis of \(L\) respecting a chosen Jordan decomposition, let \(G\) be the corresponding Gram matrix, and view \(O(L,q)\) as a group of matrices: 
\[
    O(L,q) = \{F \in \IZ_p^{r \times r} \mid FGF^\intercal = G\}.
\]
If \(F \in O(L,q)\), then the matrix \(F\) is \(a\)-approximate for any \(a\) (\Cref{def:approximate}). Conversely, we have the following identification, using \(\IZ_p/p^a\IZ_p \cong \IZ/p^a\IZ\): 
\[
    O(L/p^aL) = \{(F \bmod p^a) \in (\IZ/p^a\IZ)^{r \times r} \mid F \in \IZ_p^{r \times r}, \text{\(F\) is \(a\)-approximate}\}.
\]

Thanks to \Cref{prop:surj}, we can define \(K_0\) by the exact sequence
\[
    0 \longrightarrow  K_0 \longrightarrow O(L,q) \longrightarrow \prod_i O(\rho_i(L)) \longrightarrow 0.
\]
For \(a \geq 1\), let \(K_a\) be defined by the exact sequence
\[
    0 \longrightarrow  K_a \longrightarrow O(L,q) \longrightarrow O(L/p^aL) \longrightarrow 0.
\]

The following identifications hold, where \(I\) denotes the identity matrix:
\begin{align*}
    K_0 &= \{F \in \IZ_p^{r \times r} \mid FGF^\intercal = G, \, F_{(i,i)} \equiv I \bmod p \text{ for all } i\}, \\
    K_a &= \{F \in \IZ_p^{r \times r} \mid FGF^\intercal = G, \, F \equiv I \bmod p^a\}, \quad a \geq 1.
\end{align*}
It is then clear that we have the following subnormal series:
\[
    \dots \subseteq K_3 \subseteq K_2 \subseteq  K_1 \subseteq K_0 \subseteq O(L).
\]

To obtain generators of \(O(L/p^nL)\), it suffices to write down generators for each of the quotients of the subnormal series and lift them to precision \(p^n\) using the quadratic Hensel lifting procedure.

The proof of \Cref{prop:surj} describes how to write down generators for \(O(L)/K_0\) relying on generators of \(O(\rho_i(L))\). 
These can be inferred from the proof of \Cref{prop:order_orthogonal_elementary} and generators for the classical orthogonal and symplectic groups found in the literature \cite{generators1,generators2,generators3}.
For practical purposes, we use the generators provided by GAP \cite{GAP4}.

By the homomorphism theorem, we can view \(K_0/K_1\) as a subgroup of \(O(L/pL)\), and, more precisely, as the subgroup
\[
    K_0/K_1 = \{(F \bmod p) \in O(L/pL) \mid  \text{\(F_{(i,i)} \equiv I \bmod p\) for all \(i\)}\}.
\]
For \(p \neq 2\), \(K_0/K_1\) is just the \(p\)-group of upper unitriangular matrices with trivial block diagonal because of \cref{eq:compatible}.
For \(p = 2\), the equations cutting out this group will be described in \Cref{lem:K0}. 

Fix now \(a \geq 1\). 
Instead of \(K_a/K_{a+1}\), we write down generators for \(K_a/K_{2a}\), reducing the overall number of generators needed.
By the homomorphism theorem, we can identify \(K_a/K_{2a}\) with a subgroup of \(O(L/p^aL)\),
and, more precisely, with the subgroup \[K_a/K_{2a} = \{(F \bmod p^{2a}) \in O(L/p^{2a}L) \mid F \equiv I \bmod p^{a}\}.\]
Observe that, for any \(A,B \in \IZ_p^{r \times r}\), it holds \[(I + p^{a}A)(I + p^{a}B) \equiv I + p^{a}(A+B) \mod p^{2a}.\]
Thus, we can view \(K_a/K_{2a}\) as a free \(\IZ/p^a\IZ\)-submodule of \((\IZ/p^a\IZ)^{r \times r}\). 

In order to write down generators, we set \(F = I + p^aX\). Now, being \(2a\)-approximate gives linear conditions on \(X\). 
Essentially, there are no conditions on the upper block triangular part, while the lower block triangular part is determined by the rest and, on the diagonal part, we put what \Cref{alg:hensel_modular_odd} or \Cref{alg:hensel_modular_even} dictate. For \(p \neq 2\) or \(p = 2\) and \(a \geq 2\), we have for instance \(X_{(i,i)}= p^{i} H G_{(i,i)}^{-1}\), where \(H \in \ZZ_p^{r_i\times r_i}\) is anti-symmetric (cf. \Cref{rmk:system_odd}). 
If \(p=2\) and \(a=1\), we have to take \(A\) in a basis of the kernel of the equation solved in \Cref{lem:has_solution}. 
The entries of \(X_{(i,j)}\) with \(i<j\) can be chosen freely, whereas \(X_{(i,j)}\) for \(j<i\) is determined linearly by the equations for this matrix to be a block diagonal matrix with the same block structure as \(G\), that is, \((FGF^\intercal)_{(i,j)} = 0\) for all \(i < j\). 

\begin{example} \label{example:p=3}
Let \(p=3\) and consider the lattice \(L=\ZZ_3^3\) with diagonal Gram matrix \(G = \diag(3,9,9)\).
We compute generators of \(O(L/9L)\), which has order~\(2^4 3^5\) by \Cref{thm:order}.

Both \(\bar{q}_1\) and \(\bar{q}_2\) are non-hyperbolic \(3\)-elementary quadratic forms of dimension \(r_1 = 1\) and \(r_2 = 2\), respectively.
By \Cref{prop:order_orthogonal_elementary}, the group \(O(\rho_1(L))\times O(\rho_2(L))\) is of order \(2 \cdot 8\).
It is generated by the following matrices with coefficients in \(\IZ/3\IZ\) (all matrices are to be understood as block matrices, as indicated in the first one):
\newcommand{\rvline}{\hspace*{-\arraycolsep}\vline\hspace*{-\arraycolsep}}
\[   
    \begin{pmatrix} 
        2 &\rvline& \begin{matrix}0 & 0 \end{matrix} \\ \hline
   \begin{matrix} 0 \\ 0 \end{matrix} &\rvline 
   & \begin{matrix} 1 & 0 \\
    0 & 1 \end{matrix}
   \end{pmatrix},
 \begin{pmatrix}
    1 & 0 & 0 \\
    0 & 2 & 0 \\
    0 & 0 & 1
\end{pmatrix}, 
\begin{pmatrix}
    1 & 0 & 0 \\
    0 & 0 & 2 \\
    0 & 1 & 0
\end{pmatrix}.
\]

The quotient \(K_0/K_1\cong (\ZZ/3\ZZ)^2\) is generated by the reduction modulo~\(3\) of the following \(1\)-approximate matrices:
\[
    \begin{pmatrix}  
    1 & 1 & 0 \\
    6 & 1 & 0 \\
    0 & 0 & 1
   \end{pmatrix},
   \begin{pmatrix}  
   1 & 0 & 1 \\
   0 & 1 & 0 \\
   6 & 0 & 1
   \end{pmatrix}.
\]

Finally, \(K_1/K_2\cong (\ZZ/3\ZZ)^3\) is generated by the reduction modulo~\(9\) of the following \(2\)-approximate matrices:
\[
	\begin{pmatrix}  
		1 & 3 & 0 \\
		18 & 1 & 0 \\
		0 & 0 & 1
	\end{pmatrix},
	\begin{pmatrix}  
		1 & 0 & 3 \\
		0 & 1 & 0 \\
		18 & 0 & 1
	\end{pmatrix},
	\begin{pmatrix}  
		1 & 0 & 0 \\
		0 & 1 & 3 \\
		0 & 6 & 1
	\end{pmatrix},
\]
the \((2,2)\)-block of the last matrix coming from \Cref{alg:hensel_modular_odd} (cf. \Cref{rmk:system_odd}).

To obtain generators for \(O(L/9L)\), we plug the eight matrices into the quadratic Hensel lifting algorithm, obtain \(2\)-approximate lifts and view them modulo \(9\). 
The last three generators are already \(2\)-approximate. 
For the first five, we obtain 
\[  \begin{pmatrix}  
   8 & 0 & 0 \\
   0 & 1 & 0 \\
   0 & 0 & 1
   \end{pmatrix},
 \begin{pmatrix}
    1 & 0 & 0 \\
    0 & 8 & 0 \\
    0 & 0 & 1
\end{pmatrix}, 
\begin{pmatrix}
    1 & 0 & 0 \\
    0 & 0 & 8 \\
    0 & 1 & 0
\end{pmatrix},
\begin{pmatrix}  
   4 & 1 & 0 \\
   24 & 4 & 0 \\
   0 & 0 & 1
\end{pmatrix},
\begin{pmatrix}  
   4 & 0 & 1 \\
   0 & 1 & 0 \\
   24 & 0 & 4
   \end{pmatrix}.
\]

Looking at the image of the generators in \(O(L^\sharp)\), we obtain generators for the latter. 
It holds \(O(L^\sharp) = 2^4 3^3\), as we will confirm in \Cref{thm:discr-order}.
Indeed, only the first two generators coming from \(K_1/K_2\) belong to the kernel \(K^\sharp\) of the surjective map \(O(L/9L) \rightarrow O(L^\sharp)\).
\end{example}

The following theorem solves the problem of determining the order of \(O(L^\sharp)\). This is also useful to double-check that the generators obtained indeed generate the whole group.

\begin{theorem} \label{thm:discr-order}
Let \((L,q) = \bigoplus_{i \geq 0} (L_i, p^i q_i)\) be a Jordan decomposition of an integral \(\IZ_p\)-lattice \(L\) of rank \(r\), with \(L_i\) of rank \(r_i\). Set
\[
    w = \sum_{i>0} (i-1) \frac{r_i(r_i - 1)}{2} + \sum_{i < j} i r_i r_j .
\]
If \(p \neq 2\), then
\[
    \# O(L^\sharp) = p^{w} \prod_{i > 0} \#O (\bar{q}_i).
\]
If \(p = 2\), then
\[  
     \# O(L^\sharp) = 2^{w-t_1} \prod_{i > 0} \begin{Cases} \#O(\bar{q}_i)2^{t_0r_i} & \text{for } L_i \text{ free} \\ \#O(\bar{b}_i)2^{(t_0-1)r_i} & \text{for }L_i\text{ bound} \end{Cases} 2^{t_i-s_i}.
\]
where \(s_i\) and \(t_i\) are defined as in \Cref{thm:order}.
\end{theorem}
\begin{proof}
Let \(n\) be an integer such that \(L_{m} = 0\) for all \(m > n\). Without loss of generality, we can suppose \(n \geq 2\).
Choose a basis \(B\) of \(L\) respecting the Jordan decomposition and let \(G\) be the corresponding Gram matrix. 
Given an isometry \(f\colon L \to L\), let \(F\) be the matrix representing \(f\) with respect to \(B\). With respect to the dual basis \(B^\vee = B G^{-1}\), \(f_\QQ\) is represented by \(G^{-1}FG=F^{-\intercal}\). 
Note that the inclusion \(L \hookrightarrow L^\vee\) is represented by the rows of \(G\).
Therefore,
\[
    L^\sharp \cong \ZZ_p^r / \ZZ_p^r G \cong \prod_i (\ZZ_p/p^i \ZZ_p)^{r_i}.
\] 
Hence, the map \(L^\sharp \to L^\sharp\) induced by \(f^{-1}\) is given by \(F^\intercal\), with \((F^\intercal)_{(i,j)} = (F_{(j,i)})^\intercal\) viewed modulo \(p^j\). 
Let \(K^\sharp\) be defined by the exact sequence 
\[
    1 \longrightarrow K^\sharp \longrightarrow O(L/p^nL) \longrightarrow O(L^\sharp)\to 1
\]
and set \(K_a^\sharp = \ker(K_a/K_n \to O(L^\sharp)) \subseteq O(L/p^nL)\) for \(a \in \{0,\ldots,n\}\). 
We have 
\[
    K_1^\sharp \subseteq K_0^\sharp \subseteq K^\sharp.
\]
Thanks to \Cref{thm:order}, we only need to compute the order of \(K^\sharp\). 
Note that the order of \(K_1^\sharp\) is given by the number of \(n\)-approximate matrices \(F\) modulo \(G\ZZ_p^{r\times r}\) (since we do not transpose \(F\), we write \(G\) on the left) satisfying \(F \equiv I \bmod p\). The strictly lower triangular part of \(F\) is determined by the rest because of \cref{eq:approximate_all_p}. 

Assume that \(p \neq 2\). We have \(\#K_1^\sharp = p^a\) with \(a = \sum_{i\leq j} a_{ij}\),
where each \(a_{ij}\) is the contribution of the block \(F_{(i,j)}\) (which we view modulo \(p^i\)) of the matrix \(F\) (defined modulo \(p^n\)), given by
\[a_{ij}=\begin{cases}
    (n-1)r_0(r_0-1)/2 & \text{ for \(0 = i = j\)}, \\
    (n-1)r_0r_j & \text{ for \(0=i<j \leq n\)}, \\
    (n-i)r_i(r_i-1)/2 & \text{ for \(0 < i = j \leq n\)}, \\
    (n-i)r_ir_j & \text{ for \(0 <i < j \leq n\)}.
\end{cases}\] 
Furthermore, it holds \(\# K_0^\sharp/K_1^\sharp=p^c\) with \(c=r_0(r-r_0)\). 
By \Cref{thm:order}, the order of \(O(L/p^nL)\) is \(p^b \prod_i \#O(\rho_i(L))\) with \[b = (n-1)\frac{r(r-1)}{2} + \sum_{i<j}r_ir_j = (n-1)\sum_{i} \frac{r_i(r_i-1)}{2} + n \sum_{i<j} r_i r_j.\] 
A calculation shows that \(b - a - c = w\), with \(w\) defined as in the statement. 
Since \(K^\sharp/K_0^\sharp\cong O(\rho_0(L))\), we have the claim for \(p\neq 2\). 
 
Assume now that \(p = 2\). We set \(e_i=1\) if \(L_i\) is bound and \(e_i=0\) else.
By \Cref{thm:order}, the order of \(O(L/p^nL)\) is \(2^{b'}\prod_{i}\#O(\rho_i(L))\), with \[b' = b + \sum_{i=0}^n (t_i - e_i r_i - s_i) = b + (t_0 - r_0 t_1 + t_0t_1) + \sum_{i=1}^n(t_i - e_i r_i-s_i).\]

As before, we have \(\# K_1^\sharp=2^{a'}\) with \(a' = \sum_{i \leq j} a_{ij}'\), where \(a_{ij}'=a_{ij}\) for \(i<j\), \(a_{ii}'=a_{ii}+t_i\) for \(0\leq i \leq 1\) by \Cref{lem:has_solution}, and \(a_{ii}'=a_{ii}\) for \(i>1\).
Therefore, we have \[b' - a' = w + r_0(r-r_0)-r_0t_1+t_0t_1-t_1+\sum_{i=1}^n (t_1-e_ir_i-s_i).\]

Now, we compute the order of \(K_0^\sharp/K_1^\sharp\) which we write as \(2^{c'}\).
We write \(f\in \GL(L)\) as \(f = \bigoplus_i f_i = \bigoplus_{i,j} f_{ij}\) with \(f_i\colon L_i \to L\) and \(f_{ij}\colon L_i \to L_j \).
We write \(f \equiv f' \bmod 2^i\) if \((f-f')(L) \subseteq 2^i L\).
Putting
\[S =\{f \in O(L) \mid \text{\(f_{00} \equiv \id_{L_0} \bmod\, 2\), and \(f_{i} \equiv \id_{L_i} \bmod \, 2^i\) for \(i > 0\)}\},\]
we have \(K_0^\sharp/K_1^\sharp \cong \{f \otimes \FF_2 | f \in S\}\).
For \(f \in S\), we can write \(f_i = \id_{L_i} + 2^i \sum_{j} \hat f_{ij}\) with \(\hat f_{ij}\colon L_i \to L_j\) and set \(\bar f_{ij}=\hat f_{ij} \otimes \id_{\FF_2}\). 

In what follows, we derive linear conditions on \(\bar f_{0i}\) for \(i>0\) from the fact that \(f \in O(L)\) is an isometry. 
For any \(x_0 \in L_0\), we obtain using \(\bar f_{00}=\id\) and \(f_{0i}x_0 \in L_i\) that \[q(f_0x_0)-q(x_0)\equiv q(f_{01}x_0) \mod 2.\]
Thus, \(f_0\) being an isometry implies the following linear condition for all \(x_0 \in L_0\):
\begin{equation}\label{eq:r0t1}
    \bar{b}_1(\bar f_{01}\bar x_0,\bar v_1)=0.
\end{equation}
Similarly, we obtain for all \(i>0\) and \(x_i \in L_i\) that
\[  q(f_ix_i)-q(x_i)
= b(2^i\hat f_{ii}x_i, x_i) +2^{2i}q(\hat f_{ii}x_i)+ \sum_j q(2^i\hat f_{ij}x_i)
\equiv 2^{2i}q(\hat f_{i0}x_i) \mod 2^{2i}\]
Thus, \(f_i\) being an isometry implies that for all \(i > 0\),
\begin{equation}\label{eq:diagonal-cond}
    \bar{b}_0(\bar f_{i0}\bar x_i,\bar v_0)=0.
\end{equation}

For \(i > 0\), we have the orthogonality relation
\[0 = b(f_0x_0,f_ix_i)= \sum_k b(f_{0k}x_0, f_i x_i)= b(f_{0i}x_0, x_i) + \sum_k 2^i b( f_{0k}x_0, \hat f_{ik} x_i).\] 
Modulo \(2^{i+1}\), only the \(k=0\) term in the sum survives and we obtain 
\[b(f_{0i}x_0, x_i) + b(f_{00}x_0,f_{i0}x_i) \equiv 0 \mod 2^{i+1}.\]
Using that \(\bar f_{00}=\id\) and dividing by \(2^i\), we obtain the equivalent relation
\[\bar{b}_i(\bar f_{0i} \bar x_0, \bar x_i) =\bar{b}_0(\bar x_0,\bar f_{i0}\bar x_i).\]
We can now reformulate \cref{eq:diagonal-cond} as a condition on \(\bar f_{0i}\), namely 
\(\bar{b}_0(\bar x_i,\bar f_{0i}\bar v_0) = 0\) for all \(i>0\) and all \(x_i \in L_i\). Since \(\bar{b}_0\) is non-degenerate, this means that
\begin{equation}\label{eq:t0ri}
    \bar f_{0i} \bar v_0 = 0 \quad \text{for all } i>0.
\end{equation}

From \cref{eq:t0ri} we obtain \(r_i\) linear conditions on \(\bar f_{0i}\) if \(v_0 \neq 0\), i.e. a total of \((r-r_0)t_0\) conditions on \(f\) and from \cref{eq:r0t1}, letting \(x_0\) run through a basis of \(L_0\), another \(t_1r_0\). 
However, if both \(v_0\) and \(v_1\) are nonzero, that is \(t_0t_1=1\), then we can take \(v_0\) as part of the basis of \(L_0\) and hence \cref{eq:t0ri} implies one of the \(t_1r_0\) equations of \cref{eq:r0t1}. 
Therefore, we have \(c' = r_0(r - r_0) - t_0 (r-r_0) - t_1r_0 + t_0t_1\) and
\[b'-a'-c'=w  -t_1 + t_0(r-r_0)+ \sum_{i=1}^n(t_i - e_i r_i-s_i)= w -t_1 + \sum_{i=1}^n(t_i + (t_0-e_i) r_i-s_i).\]

This yields the claim for \(p = 2\) and concludes the proof.
Note that \(b'-a'-c'\) does not depend on \(r_0\) anymore, as it should. 
\end{proof}

The following lemma describes how to obtain generators for \(K_0/K_1\).

\begin{lemma} \label{lem:K0}
Let \(L = \bigoplus_i (L_i,2^iq_i)\) be a Jordan decomposition of an integral \(\IZ_2\)-lattice. 
In the space of compatible maps \(f = \bigoplus_{ij} 2^{\max(0,j-i)}f_{ij}\) with \(f_{ii} = \id_{L_i}\), the group \(K_0/K_1\) is cut out by the following equations modulo~\(2\), which holds for all \(i<j\) and \(x_i \in L_i\), \(x_j \in L_j\):
    \begin{gather}
        \sum_{k=i}^j b_k(f_{ik}x_i,f_{jk}x_j) \equiv 0 \label{eq:1-adaptedK0a} \\ 
        b_{i}(x_i,f_{i-1,i}v_{i-1}) + b_{i+1}(f_{i,i+1}x_i,v_{i+1}) \equiv 0 \label{eq:1-adaptedK0b} \\
        b_{i-2}(f_{i,i-2}v_i,v_{i-2}) + q_{i-1}(f_{i,i-1}v_i) + q_{i+1}(f_{i,i+1}v_i) + b_{i+2}(f_{i,i+2}v_i,v_{i+2}) \equiv 0 \label{eq:1-adaptedK0c}
    \end{gather}
\end{lemma}
\begin{proof}
Write \(f = \bigoplus f_i\) with \(f_i\colon L_i \to L\). Then, \(f\) is \(1\)-approximate if and only if for all \(i,j\) and \(x_i \in L_i\), \(x_j \in L_j\) it holds that \(b(f_i x_i, f_j x_j) \equiv 0 \bmod 2^{1+\max(i,j)}\), \(q(x_i)\equiv q(f_i x_i) \bmod 2\) and \(q(v_i) \equiv q(f_i v_i) \bmod 4\).
    
Writing \(f_i = \bigoplus_j 2^{\max(0,j-i)}f_{ij}\), we can reformulate these as conditions modulo~\(2\) on the \(f_{ij}\). For \cref{eq:1-adaptedK0a}, this is just a matter of collecting powers of~\(2\). We have
\begin{align*} 
    q(f_i(x_i)) 
    &\equiv \sum_{j} 2^{|j-i|}q_j(f_{ij}x_i) \mod 2 \\
    &\equiv 2 q_{i-1}(f_{i,i-1}x_i) +q_i(f_{ii}x_i)+ 2q_{i+1}(f_{i,i+1}x_i) \mod 2 \\
    &\equiv b_{i-1}(f_{i,i-1}x_i,v_{i-1}) +q_i(x_i)+ b_{i+1}(f_{i,i+1}x_i,v_{i+1}) \mod 2
\end{align*}
Further, \(b(f_{i-1}x_{i-1},f_ix_i) = 0\) yields \(b_{i-1}(x_{i-1},f_{i,i-1}x_i) \equiv b_{i}(x_{i},f_{i-1,i}x_{i-1}) \bmod 2\) and \cref{eq:1-adaptedK0b} follows from \(q(f_i(x_i)) = q_i(x_i)\).
A similar calculation shows that \(q(f_{i}v_i) \equiv q(v_i) \bmod 2^{i+2}\) yields \cref{eq:1-adaptedK0c}.
\end{proof}
\begin{remark}
     \Cref{eq:1-adaptedK0b} implies that \(b_i(v_i,f_{i-1,i}v_{i-1})\equiv b_{i-1}(f_{i,i-1}v_i,v_{i-1}) \equiv 0 \bmod 2\) for all \(i\). Indeed, inserting \(x_i=v_i\) in \cref{eq:1-adaptedK0b} yields that \(b_i(v_i,f_{i-1,i}v_{i-1}) = b_{i+1}(f_{i,i+1}v_i,v_{i+1}) \bmod 2\). Note that the left hand and the right hand side differ just by an index shift. But for \(i\) high enough, \(v_{i+1}=0\), which yields the claim. 
     
     In particular the middle contributions in \cref{eq:1-adaptedK0c} are integral, so that it is an equation over \(\ZZ_2/2\ZZ_2\).
\end{remark}

\begin{example} \label{example:p=2}
Let \(p = 2\) and consider the lattice \(L=\ZZ_2^4\) with diagonal Gram matrix \(G=\diag(1,2,2,4)\). 
We compute generators of \(O(L/4L)\) and of \(O(L^\sharp)\). 

With the notation of \Cref{thm:order}, we have \(r_0 = 1\), \(r_1 = 2\), \(r_2 = 1\), \(t_0 = t_1 = t_2 = 1\), \(s_0 = -1\), \(s_1 = 1\), and all other \(r_i\), \(t_i\), \(s_i\) are equal to \(0\). 
Moreover, it holds \(\bar{b}_0 = \bar{b}_2 = \bar \bw\) and \(\bar{b}_1 = \bar \bw^{\oplus 2}\).
Therefore, we can compute \(\#O(L/2L) = 2^2\) and \(\#O(L/4L) = 2^{11}\).
Indeed, \(O(L/2L)\) is generated by the reduction modulo~\(2\) of the matrices
\[
\begin{pmatrix}
    1 & 0 & 0 & 0\\
    0 & 0 & 1 & 0\\
    0 & 1 & 0 & 0\\
    0 & 0 & 0 & 1
\end{pmatrix},
\begin{pmatrix}
    1 & 3 & 3 & 1\\
    2 & 1 & 2 & 1\\
    2 & 2 & 1 & 1\\
    0 & 2 & 2 & 1
\end{pmatrix},
\]
the second one coming from \Cref{lem:K0}. In \(K_1/K_2\), we have nine generators whose reductions modulo \(4\) are equal to \(I + 2 E_{ij}\), with \(1\leq i \leq j \leq 4\) and \((i,j)\neq (2,3)\). 

By \Cref{thm:discr-order}, the order of \(O(L^\sharp)\) is \(2^3\). Indeed, all generators found belong to \(K^\sharp\), except the three generators coming from \(O(L/2L)\) or from \(I+2E_{4,4}\).
\end{example}

\begin{remark} 
We implemented the formulas from \Cref{thm:order,thm:discr-order} in a computer and computed generators for \(O(L/p^nL)\) and \(O(L^\sharp)\). 
Then, we compared the order of the groups generated with our formulas. 
We checked that they agree for all \(2\)-adic lattices of rank at most \(6\), determinant dividing \(2^{10}\) and \(n \in \{1,2\}\). 
This provides a sanity check that the formulas displayed are indeed correct and that we do not miss any generators in our implementation.
\end{remark}

\bibliographystyle{alpha}
\bibliography{literature}

\end{document}